\documentclass{amsart}
\usepackage{amsmath,amsthm}
\usepackage{amsfonts,amssymb}
\usepackage[hidelinks]{hyperref}
\usepackage{enumerate}
\usepackage{ bbold }
\usepackage[english]{babel}
\usepackage[parfill]{parskip}
  
\usepackage{pdflscape}

\newtheorem{theorem}{Theorem}
\newtheorem{proposition}[theorem]{Proposition}
\newtheorem{lemma}[theorem]{Lemma}
\newtheorem{corollary}[theorem]{Corollary}
\newtheorem{definition}[theorem]{Definition}

\numberwithin{theorem}{section}
\numberwithin{figure}{section}
\numberwithin{table}{section}

\usepackage{float}
\usepackage{caption}

\newcommand{\br}{0.6}
\newcommand{\h}{1}
\newcommand{\sca}{1}
\usepackage{tikz}
\usepackage{tikz-cd}
\usepackage{xparse}
\usepackage{listofitems}

\newcommand\z[2]{%
  \readlist*\student{#1}%
  First argument: #2.\par
  \ifnum\studentlen=1\relax%
    Only one name provided: \student[1]%
  \else%
    Two names provided: \student[1] and \student[2]%
  \fi%
}

\newcommand{\di}[2]{
\readlist*\scalesarguments{#1}%
\ifnum\scalesargumentslen=1\relax
 \begin{tikzpicture}[baseline={(current bounding box.center)},scale={\scalesarguments[1]}]
 {#2}
 \end{tikzpicture}
  \else
\begin{tikzpicture}[baseline={(current bounding box.center)},xscale={\scalesarguments[1]},yscale={\scalesarguments[2]}]
{#2}
\end{tikzpicture}
  \fi%
}

\NewDocumentCommand{\li}{ooo}{%
  \IfValueTF{#1}{%
	  \IfValueTF{#3}{
	  \draw [{#3}] ({#1},{#2}) to ({#1},{#2+\h});
	  }
	  {
	\draw ({#1},{#2}) to ({#1},{#2+\h});
  }
  }
  {%
    \draw (0,0) to (0,\h);
  }%
}

\NewDocumentCommand{\lin}{ooooo}{%
  \IfValueTF{#1}{%
	  \IfValueTF{#5}{
	  \draw [{#5}] ({#1},{#2}) to ({#3},{#4});
	  }
	  {
	\draw ({#1},{#2}) to ({#3},{#4});
  }
  }
  {%
    \draw (0,0) to (0,\h);
  }%
}

\NewDocumentCommand{\lili}{ooo}{%
  \IfValueTF{#1}{%
\IfValueTF{#3}{
	\li[{#1}][{#2}][{#3}]
	\li[{#1+\br}][{#2}][{#3}]
	}{
	\li[{#1}][{#2}]
	\li[{#1+\br}][{#2}]
	}
  }{%
    \li 
	\li[{\br}][0]
  }%
}

\NewDocumentCommand{\dli}{ooo}{%
  \IfValueTF{#1}{%
  \IfValueTF{#3}{
	  \draw[{#3}] ({#1},{#2}) to ({#1},{#2-0.4*\h});
  }{
	\draw ({#1},{#2}) to ({#1},{#2-0.4*\h});
	}
  }{%
    \draw (0,0) to (0,-0.4*\h);
  }%
}

\NewDocumentCommand{\uli}{ooo}{%
  \IfValueTF{#1}{%
  \IfValueTF{#3}{
	  \draw[{#3}] ({#1},{#2}) to ({#1},{#2+0.4*\h});
  }
  {
	\draw ({#1},{#2}) to ({#1},{#2+0.4*\h});
	}
  }{%
    \draw (0,0) to (0,0.4*\h);
  }%
}

\NewDocumentCommand{\ocr}{ooo}{%
  \IfValueTF{#1}{%
    \IfValueTF{#3}{
    \draw[{#3}] ({#1},{#2}) to ({#1+\br},{#2+\h});
	\draw [{#3}]({#1+\br},{#2}) to  ({#1+3*\br/4},{#2+\h/4});
	\draw[dotted,{#3}]({#1+3*\br/4},{#2+\h/4}) to  ({#1+\br/4},{#2+3*\h/4});
	\draw[{#3}] ({#1+\br/4},{#2+3*\h/4}) to  ({#1},{#2+\h});
    }
    {
	\draw ({#1},{#2}) to ({#1+\br},{#2+\h});
	\draw ({#1+\br},{#2}) to  ({#1+3*\br/4},{#2+\h/4});
	\draw[dotted]({#1+3*\br/4},{#2+\h/4}) to  ({#1+\br/4},{#2+3*\h/4});
	\draw ({#1+\br/4},{#2+3*\h/4}) to  ({#1},{#2+\h});
	}
  }{%
    \draw (0,0) to (\br,\h);
	\draw (\br,0) to  (3*\br/4,\h/4);
	\draw[dotted](3*\br/4,\h/4) to  (\br/4,3*\h/4);
	\draw (\br/4,3*\h/4) to  (0,\h);

  }%
}

\NewDocumentCommand{\cro}{ooo}{%
  \IfValueTF{#1}{%
    \IfValueTF{#3}{
    \draw[{#3}] ({#1},{#2}) to ({#1+\br},{#2+\h});
	\draw [{#3}]({#1+\br},{#2}) to  ({#1+3*\br/4},{#2+\h/4});
	\draw[{#3}]({#1+3*\br/4},{#2+\h/4}) to  ({#1+\br/4},{#2+3*\h/4});
	\draw[{#3}] ({#1+\br/4},{#2+3*\h/4}) to  ({#1},{#2+\h});
    }
    {
	\draw ({#1},{#2}) to ({#1+\br},{#2+\h});
	\draw ({#1+\br},{#2}) to  ({#1+3*\br/4},{#2+\h/4});
	\draw({#1+3*\br/4},{#2+\h/4}) to  ({#1+\br/4},{#2+3*\h/4});
	\draw ({#1+\br/4},{#2+3*\h/4}) to  ({#1},{#2+\h});
	}
  }{%
    \draw (0,0) to (\br,\h);
	\draw (\br,0) to  (3*\br/4,\h/4);
	\draw (3*\br/4,\h/4) to  (\br/4,3*\h/4);
	\draw (\br/4,3*\h/4) to  (0,\h);

  }%
}

\NewDocumentCommand{\ucr}{ooo}{%
  \IfValueTF{#1}{%
     \IfValueTF{#3}{
     \draw[{#3}] ({#1},{#2+\h}) to ({#1+\br}, {#2});
	\draw[{#3}] ({#1},{#2}) to  ({#1+\br/4},{#2+\h/4});
	\draw[dotted,#3]({#1+\br/4},{#2+\h/4}) to  ({#1+3*\br/4},{#2+3*\h/4});
	\draw[{#3}]({#1+3*\br/4},{#2+3*\h/4}) to  ({#1+\br},{#2+\h});
     }
     {
	\draw ({#1},{#2+\h}) to ({#1+\br}, {#2});
	\draw ({#1},{#2}) to  ({#1+\br/4},{#2+\h/4});
	\draw[dotted]({#1+\br/4},{#2+\h/4}) to  ({#1+3*\br/4},{#2+3*\h/4});
	\draw ({#1+3*\br/4},{#2+3*\h/4}) to  ({#1+\br},{#2+\h});
	}
  }{%
    \draw (0,\h) to (\br, 0);
	\draw (0,0) to  (\br/4,\h/4);
	\draw[dotted](\br/4,\h/4) to  (3*\br/4,3*\h/4);
	\draw (3*\br/4,3*\h/4) to  (\br,\h);

  }%
}

\NewDocumentCommand{\cu}{ooo}{%
  \IfValueTF{#1}{%
       \IfValueTF{#3}{
       \draw[{#3}]  ({#1},{#2}) to [out=-90, in=180] +(\br/2,-0.4*\h);
	\draw [{#3}] ({\br+#1}, {#2}) to [out=-90, in=0] +(-\br/2,-0.4*\h);
       }
       {
	\draw ({#1},{#2}) to [out=-90, in=180] +(\br/2,-0.4*\h);
	\draw ({\br+#1}, {#2}) to [out=-90, in=0] +(-\br/2,-0.4*\h);
	}
  }{%
    \draw (0,0) to [out=-90, in=180] +(\br/2,-0.4*\h);
	\draw (\br, 0) to [out=-90, in=0] +(-\br/2,-0.4*\h);

  }%
}

\NewDocumentCommand{\ca}{ooo}{%
  \IfValueTF{#1}{%
         \IfValueTF{#3}{
	\draw[{#3}]  ({#1},{#2}) to [out=90, in=180] +(\br/2,0.4*\h);
	\draw[{#3}]  ({\br+#1}, {#2}) to [out=90, in=0] +(-\br/2,0.4*\h);         
         }
	{         
	\draw ({#1},{#2}) to [out=90, in=180] +(\br/2,0.4*\h);
	\draw ({\br+#1}, {#2}) to [out=90, in=0] +(-\br/2,0.4*\h);
	}
  }{%
    \draw (0,0) to [out=90, in=180] +(\br/2,0.4*\h);
	\draw (\br, 0) to [out=90, in=0] +(-\br/2,0.4*\h);

  }%
}

\NewDocumentCommand{\cc}{ooo}{%
  \IfValueTF{#1}{%
           \IfValueTF{#3}{         
	\draw[{#3}] ({#1},{#2}) to [out=90, in=180] +(\br/2,0.4*\h);
	\draw[{#3}] ({\br+#1}, {#2}) to [out=90, in=0] +(-\br/2,0.4*\h);
	\draw [{#3}]({#1},{#2+\h}) to [out=-90, in=180] +(\br/2,-0.4*\h);
	\draw [{#3}]({\br+#1}, {#2+\h}) to [out=-90, in=0] +(-\br/2,-0.4*\h);
           }         
           {
  	\draw ({#1},{#2}) to [out=90, in=180] +(\br/2,0.4*\h);
	\draw ({\br+#1}, {#2}) to [out=90, in=0] +(-\br/2,0.4*\h);
	\draw ({#1},{#2+\h}) to [out=-90, in=180] +(\br/2,-0.4*\h);
	\draw ({\br+#1}, {#2+\h}) to [out=-90, in=0] +(-\br/2,-0.4*\h);
			}
  }{%
    \draw (0,0) to [out=90, in=180] +(\br/2,0.4*\h);
	\draw (\br, 0) to [out=90, in=0] +(-\br/2,0.4*\h);
	\draw (0,\h) to [out=-90, in=180] +(\br/2,-0.4*\h);
	\draw (\br, \h) to [out=-90, in=0] +(-\br/2,-0.4*\h);

  }%
}

\NewDocumentCommand{\culi}{ooo}{%
  \IfValueTF{#1}{%
       \IfValueTF{#3}{  
	\draw[{#3}] ({#1},{#2}) to [out=-90, in=180] +(\br/2,-0.4*\h);
	\draw[{#3}]({\br+#1}, {#2}) to [out=-90, in=0] +(-\br/2,-0.4*\h);
	\draw[{#3}] ({2*\br+#1}, {#2}) to ({2*\br+#1}, {-\h + #2});
       
       }
       {
	\draw ({#1},{#2}) to [out=-90, in=180] +(\br/2,-0.4*\h);
	\draw ({\br+#1}, {#2}) to [out=-90, in=0] +(-\br/2,-0.4*\h);
	\draw ({2*\br+#1}, {#2}) to ({2*\br+#1}, {-\h + #2});
		}
  }{%
    \draw (0,0) to [out=-90, in=180] +(\br/2,-0.4*\h);
	\draw (\br, 0) to [out=-90, in=0] +(-\br/2,-0.4*\h);
	\draw ({2*\br}, 0) to ({2*\br}, -\h);

  }%
}

\NewDocumentCommand{\licu}{ooo}{%
  \IfValueTF{#1}{%
         \IfValueTF{#3}{  
	\draw[{#3}] ({\br+#1}, {#2}) to [out=-90, in=180] +(\br/2,-0.4*\h);
	\draw[{#3}]  ({2*\br+#1}, {#2}) to [out=-90, in=0] +(-\br/2,-0.4*\h);
	\draw[{#3}] ({#1},{#2}) to ({#1}, {-\h + #2});
         
         }
         {
	\draw ({\br+#1}, {#2}) to [out=-90, in=180] +(\br/2,-0.4*\h);
	\draw  ({2*\br+#1}, {#2}) to [out=-90, in=0] +(-\br/2,-0.4*\h);
	\draw ({#1},{#2}) to ({#1}, {-\h + #2});
		}
  }{%
    \draw (\br, 0) to [out=-90, in=180] +(\br/2,-0.4*\h);
	\draw  ({2*\br}, 0)  to [out=-90, in=0] +(-\br/2,-0.4*\h);
	\draw (0,0)to ({0}, -\h);

  }%
}

\NewDocumentCommand{\ccr}{ooo}{%
  \IfValueTF{#1}{%
   \IfValueTF{#3}{  
   \cu[{#1}][{#2-\h}][{#3}]    
    \li[{#1}][{#2-\h}][{#3}]
	\ocr[{#1+\br}][{#2-\h}][{#3}]
	\dli[{#1+2*\br}][{#2-\h}][{#3}]
   }
   {
	\cu[{#1}][{#2-\h}]    
    \li[{#1}][{#2-\h}]
	\ocr[{#1+\br}][{#2-\h}]
	\dli[{#1+2*\br}][{#2-\h}]
	}
  }{%
	\cu[0][-\h]    
    \li[0][-\h]
	\ocr[\br][-\h]
	\dli[2*\br][-\h]
  }%
}

\NewDocumentCommand{\cali}{ooo}{%
  \IfValueTF{#1}{%
     \IfValueTF{#3}{  
     \ca[{#1}][{#2}][{#3}]
	\li[{#1+2*\br}][{#2}][{#3}]
     }
     {
	\ca[{#1}][{#2}] 
	\li[{#1+2*\br}][{#2}]
	}
  }{%
    \ca 
    \li[2*\br][0]
  }%
}

\NewDocumentCommand{\lica}{ooo}{%
  \IfValueTF{#1}{%
       \IfValueTF{#3}{  
       }
	\li[{#1}][{#2}][{#3}]
    \ca[{#1+\br}][{#2}][{#3}]
       {
	\li[{#1}][{#2}]
    \ca[{#1+\br}][{#2}]
    }
  }{%
    \li 
    \ca[\br][0]
  }%
}

\NewDocumentCommand{\crc}{ooo}{%
  \IfValueTF{#1}{%
         \IfValueTF{#3}{  
	\uli[{#1}][{#2+\h}][{#3}]
    \li[{#1+2*\br}][{#2}][{#3}]
	\ocr[{#1}][{#2}][{#3}]
	\ca[{#1+\br}][{#2+\h}][{#3}]
         
         }
         {
	\uli[{#1}][{#2+\h}]    
    \li[{#1+2*\br}][{#2}]
	\ocr[{#1}][{#2}]
	\ca[{#1+\br}][{#2+\h}]
	}
  }{%
	\uli[0][\h]    
    \li[2*\br][0]
	\ocr
	\ca[\br][\h]
  }%
}

\usetikzlibrary{shapes.geometric}
\usetikzlibrary{matrix, calc, arrows}

\usepackage{dsfont}
\usepackage{mathtools}

\DeclareMathOperator{\End}{End}
\DeclareMathOperator{\Hom}{Hom}

\DeclareMathOperator{\id}{{\rm id}}

\DeclareMathOperator{\op}{op}

\newcommand{\mK}{\mathbb{K}}

\newcommand{\mC}{\mathbb{C}}

\newcommand{\cC}{\mathcal{C}}
\newcommand{\cD}{\mathcal{D}}

\newcommand{\oa}{\bar{0}}
\newcommand{\ob}{\bar{1}}

\renewcommand{\epsilon}{\varepsilon}
\newcommand{\parS}{\varsigma}
\newcommand{\parq}{\rho}

\DeclarePairedDelimiter\abs{\lvert}{\rvert}%
\DeclarePairedDelimiter\norm{\lVert}{\rVert}%
\makeatletter
\let\oldabs\abs
\def\abs{\@ifstar{\oldabs}{\oldabs*}}
\let\oldnorm\norm
\def\norm{\@ifstar{\oldnorm}{\oldnorm*}}
\makeatother

\makeatletter
\@namedef{subjclassname@2020}{\textup{2020} Mathematics Subject Classification}
\makeatother

\renewcommand{\arraystretch}{1.5} 

\begin{document}
\title{Diagram categories of Brauer type}

\author{Sigiswald Barbier}
\address{Department of Electronics and Information Systems \\Faculty of Engineering and Architecture\\Ghent University\\Krijgslaan 281, 9000 Gent\\ Belgium.}
\email{barbier.sigiswald@gmail.com}

\date{\today}

\begin{abstract}
This paper introduces monoidal (super)categories resembling the Brauer category.
For all categories, we can construct bases of the hom-spaces using Brauer diagrams.  
These categories include the Brauer category, its deformation the BWM-category, the periplectic Brauer category, and its deformation the periplectic $q$-Brauer category but also some new exotic categories. We show that the BWM-category is the unique deformation of the Brauer category in this framework, while the periplectic Brauer category has two deformations, which are each other monoidal opposite.
\end{abstract}

\maketitle

\section{Introduction}

\subsection{Diagram algebras and categories}
It is well known that certain classes of algebras such as the symmetric group algebra, the Temperley-Lieb algebra, the Brauer algebra, and the periplectic Brauer algebra all can be graphically represented using diagrams.
This diagrammatical interpretation allows us to construct for each class of these diagram algebras a corresponding category. The objects of such category $\cC$ are the natural numbers and homomorphisms between $m$ and $n$ in $\cC$ are given by linear combinations of the appropriate  $(m,n)$-diagrams. The diagram algebras themselves can then be recovered by the endomorphism algebras $\End_\cC(m,m)$. This categorical approach has many advantages \cite{LehrerZhangDiagramCategories, SamSnowden}. 
For example, Lehrer and Zhang \cite{LehrerZhang} used the Brauer category associated with the Brauer algebras to prove and extend the first and second theorem of invariant theory for the orthogonal and symplectic groups in an elegant manner.  

In this paper, we will be mainly interested in the Brauer and periplectic Brauer categories. Schur-Weyl duality relates the Brauer algebra to the symplectic group, the orthogonal group \cite{Brauer}, or the encompassing orthosymplectic supergroup \cite{BenkartLeeRam} while the periplectic Brauer algebra is related via Schur-Weyl duality to the periplectic Lie supergroup \cite{Moon}.
The Brauer algebra and the periplectic Brauer algebra are at first glance very similar. They are depicted using the same diagrams and the multiplication rules are equivalent up to some minus signs. This makes it possible to describe them together in one framework as done by Kujawa and Tharp \cite{KujawaTharp}.
The similarity between these categories leads to similarities in their representation theory. This manifests itself, for instance, in the classification and labelling of the simple modules \cite[Theorem 4.3.1]{KujawaTharp}. 
However, some properties can still differ wildly between these categories. For example, the classification of the blocks is completely different for the Brauer algebra compared to the periplectic Brauer algebra \cite{CDM,CoulembierPeriplectic}. 

\begin{figure}[h]
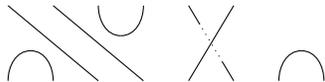

\begin{center}
\[
\di{\sca}{\ca \cu[2*\br][\h] \ocr[4*\br][0] \ca[6*\br][0] \lin[2*\br][0][0][\h] \lin[3*\br][0][\br][\h] }
\]
\caption{An example of a $(8,6)$-Brauer diagram with one cup, two caps, and four propagating lines.}\label{diagex}
\end{center}
\end{figure}

The Brauer algebra and the periplectic Brauer algebra also have deformations, namely the Birman-Wenzl-Murakami algebra \cite{BirmanWenzl, Murakami} and the periplectic  $q$-Brauer algebra \cite{AGG, RuiSong}. These deformed algebras can also be represented using diagrams and have corresponding categories.

The Brauer category, the BWM-category, and the periplectic ($q$-)Brauer category all share some interesting properties. They are monoidally generated by one object and three morphisms: $\di{0.5\sca}{\ocr}$, $\di{0.5\sca}{\cu}$ and $\di{0.5\sca}{\ca}$. In the ($q$-)periplectic case these last two morphisms are required to be odd so we obtain a graded (i.e.\ super) category. For each category, we can use the set of Brauer diagrams as a basis for the $\hom$-spaces. Figure \ref{diagex} gives an example of such a Brauer diagram.  
These categories also have a natural triangular decomposition. This decomposition allows us to define in each case so-called standard modules or cell modules, which are a great tool in the study of the representation theory \cite{BrundanStroppel, CoulembierZhang}.

The characterising difference to distinguish between these categories is in the relations the morphisms satisfy.  Figure \ref{Figure1} shows some examples of such relations. The research behind this article grew out of the search for the defining relations for a deformation of the periplectic Brauer category.  Recently, Rui and Song \cite{RuiSong} defined such a deformation, which they call the periplectic $q$-Brauer category.
Although no motivation for the choice of generators and relations is given, we know that they are the `correct’ ones because they lead to a ­periplectic $q$-Brauer algebra which satisfies a Schur­-Weyl duality with the quantum group of type P \cite{AGG}. An important motivation behind this paper was to understand better why exactly these relations are the correct ones. 

We will see in Section \ref{Subsection q periplectic} that there exist two categories satisfying the assumption that it is a diagram category of Brauer type as defined in Definition \ref{Definition category of Brauer type} and which has the periplectic Brauer category as a limit. One of these categories is isomorphic to the periplectic $q$-Brauer category defined in \cite{RuiSong} and the other is its monoidal opposite (in the sense of \cite[Definition 2.1.5]{TensorCategories}). In some way, we can thus conclude that the periplectic $q$-Brauer category is the only possible deformation of the periplectic Brauer category.
In Section \ref{Subsection BWM category}, we will also show that the BWM-category is the unique possible deformation for the Brauer category. The BWM-category is its own  monoidal opposite, see Corollary \ref{Corollary Monoidal opposites}.

\subsection{The diagram categories of Brauer type}
In this paper, we look at the following problem. Which monoidal (super)categories exist which are generated by one object and three morphisms, $H=\di{0.5\sca}{\ocr}$,  $A=\di{0.5\sca}{\ca}$ and $U=\di{0.5\sca}{\cu}$, where $A$ and $U$ have the same parity, if we furthermore require that $H$ is invertible and induces a braiding and the $\hom$-spaces can be given a basis using the same diagrams as for the Brauer category. 
This last condition forces us to impose relations, such as in Figure \ref{Figure1}, to simplify diagrams. 

\begin{figure}
\begin{center}
\[
 	\di{\sca}{\ca \cu[\br][0] \dli \uli[2*\br][0]} = \parS \; \di{\sca}{\li}, \quad \di{\sca}{\ocr \cu } = \lambda\; \di{\sca}{\cu}, \quad \di{\sca}{\ucr} = \alpha \; \di{\sca}{\lili} + \beta \; \di{\sca}{\ocr} +\gamma \; \di{\sca}{\cc}.
\]
\caption{Some examples of relations to simplify diagrams. \label{Figure1}}
\end{center}
\end{figure}

A priori imposing such relations introduces a whole plethora of parameters and corresponding categories. Remarkably, we only get a fairly small list of possible categories. We summarized the possible categories in Table \ref{summary table}. These categories have at most four independent parameters. By a rescaling isomorphism, we can reduce the number of independent parameters further to at most two. 
We can distinguish between the different categories using only four relations: sliding, upside-down sliding, straightening, and untwisting. 
Let us now take a closer look at these distinguishing relations. 

The main relation to distinguish between categories is the sliding relation and its upside-down version:
\begin{align*}
\di{0.5\sca}{ \ocr \dli \cu[\br][0] \li[2*\br][0] } =d\; \di{0.5\sca}{\culi}  + e\; \di{0.5\sca}{\ccr} +f\;  \di{0.5\sca}{\licu }\quad \text{ and } \quad \di{0.5\sca}{\crc} =d'\; \di{0.5\sca}{\cali} + e'\; \di{0.5\sca}{ \li \ocr[\br][0] \ca[0][\h] \dli} +f'\; \di{0.5\sca}{\lica} \; .
\end{align*}
The parameter $e$ always satisfies $e^4=1$. For most categories, we got a stronger condition. We have either $e^2=1$ or $e^2=-1$, where the odd case always has a different sign than the even case. In particular, if $-1$ is not a square in our field $\mK$, only the even category or the odd supercategory exists, but not both. 

We also distinguish between categories by whether the values for $f$ and $d$ are zero. If they are both non-zero, they are related by $d=-ef$. Note that by rescaling we can always set $f$ or $d$ to $1$, if they are non-zero. 
For most categories, the value of $e'$ is determined by $e$, but there are a few cases where it is independent. In these cases $e'$ satisfies the same constraints as $e$. The parameters $d'$ and $f'$ can always expressed using the other parameters. Namely, we always have $d=-ef'$ and $d'=-e'f$. 

The straightening relation $\di{\sca}{\ca \cu[\br][0] \dli \uli[2*\br][0]} = \parS \; \di{\sca}{\li}$ is another relation dividing categories into two different cases. We have categories where $\parS $ is zero and categories for which $\parS $ is non-zero. If $\parS $ is non-zero, we can use another rescaling to put $\parS =1$. 

The last relation to divide between categories is the untwisting relation:
$\di{0.5\sca}{
\ocr \ca[0][\h]
}= \lambda' \; \di{0.5\sca}{\ca}$.
Here we have that either the parameter $\lambda'$ has the same value as the parameter $\lambda$ determined by $\di{0.5\sca}{
\ocr 
\cu
}
 = \lambda  \;\di{0.5\sca}{
 \cu
} $ or that the value of $\lambda'$ is different from $\lambda$. Even if $\lambda'\not=\lambda$, the value of $\lambda'$ is not independent but can be expressed using other parameters. 

We will use the notation $\cC^{f,f'}_{\lambda',\parS }(\epsilon, e,e')$ for a category with the corresponding values for the parameters $f, f',\lambda',\parS ,e,e'$ and where $\epsilon$ is $+$ or $-$ depending on the parity of the cup and cap. For example $\cC^{b,0}_{b-\lambda,0}(+,i,-i)$  is the monoidal category with $f=b$, $f'=0$, $\lambda'=b-\lambda$, $\parS =0$, $e=i$, $e'=-i$ and where the cup and cap are even morphisms. 
Remark that not all possible values for the parameters lead to a valid category. 
For instance, the category $\cC^{b,0}_{b-\lambda,0}(-,i,-i)$ does not exist. See Table \ref{summary table} for the allowed categories. 

In Section \ref{Section connection with existing categories}, we will show that the Brauer category is given by $\cC^{0,0}_{1,1}(+,1,1)$, the Birman-Wenzl-Murakami category is  $\cC^{z,z}_{v,1}(+,1,1)$, the periplectic Brauer category is  $\cC^{0,0}_{-1,1}(-,1,1)$  and the periplectic $q$-Brauer category is  $\cC^{q-q^{-1},0}_{-q^{-1},1}(-,1,1)$.

\subsection{Structure of the paper}
We start the paper with two preliminary sections. In Section \ref{Section Brauer diagrams} we define Brauer diagrams and show how we can represent them graphically. We introduce fundamental diagrams. We also define a standard expression for each Brauer diagram. This is a unique decomposition of a Brauer diagram into fundamental diagrams. It is these standard expressions that we will use as a basis for the hom-spaces of the categories of Brauer type. 
Since the categories of Brauer type will be monoidally generated supercategories, we
recall some basic facts about monoidal supercategories in Section \ref{Section monoidal supercategories}.

In Section \ref{Section motivating}, we define what we mean by a diagram category of Brauer type, while in Section \ref{Section relations} we give some relations to simplify diagrams. We show in Theorem \ref{theorem simplifying} that these relations are sufficient to reduce any morphism in our category to a linear combination of standard expressions. 

Section \ref{Section rewritable diagrams} derives the equations the parameters have to satisfy to obtain a well-defined category. Most of the calculations are redelegated to Appendix \ref{Appendix rewrite equations} to not distract from the flow of the story.
We then solve these equations in Section \ref{Section classes of categories} leading to an overview in Table \ref{summary table} of the possible categories of Brauer type. Theorem \ref{Theorem basis} in Section \ref{Section basis} establishes that the standard expressions are linearly independent and thus that the  Brauer diagrams indeed form a basis for the hom-spaces.  

In Section \ref{Section scaling}, we introduce functors giving isomorphisms between different categories of Brauer type. This allows us in particular to rescale some independent parameters. 
We end the paper by giving the relation between the categories of Brauer type introduced in this paper and existing categories and algebras in the literature in Section \ref{Section connection with existing categories}. We also show that the Brauer category has a unique deformation, while there exist two possible deformations for the periplectic Brauer category. 
As an aside, we prove that the $q$-Brauer algebra introduced by Wenzl in \cite{Wenzl} does not fit in our framework.

In this paper, algebras and linear structures will be defined over a field $\mK$ of characteristic different from $2$. Most results still hold if we take $\mK$ to be an integral domain.
We will also introduce several parameters $\alpha_1, \dots, \alpha_n$. These parameters can either be seen as elements in $\mK$ or as formal variables. In the latter case, we will work over the ring $\mK[\alpha_1, \dots,\alpha_n]$ or even $\mK[\alpha_1, \dots,\alpha_n, \alpha_1^{-1},\dots \alpha_n^{-1}]$ if the parameters are assumed to be invertible.

\section{Brauer diagrams} \label{Section Brauer diagrams}
In this section, we will introduce Brauer diagrams and define for each Brauer diagram a unique way to depict it graphically, which we will call the standard expression for that Brauer diagram. It will be these standard expressions that we will use in this paper as a basis for the hom-spaces of our categories. 
\begin{definition}[Brauer diagram]
An $(r,s)$-Brauer diagram is a partitioning of $r+s$ dots into disjunct pairs. 
\end{definition}

We will depict such a Brauer diagram graphically in a diagram, see Figure \ref{diagex}. We draw $r$-bottom dots on a horizontal line, $s$-top dots on another horizontal line above the first and connect the paired dots with arcs. An arc connecting two bottom dots is called a cap, while an arc connecting two top dots is called a cup. An arc connecting a bottom top and top dot we call a propagating line. 
\begin{definition} \label{Def fundamental diagrams}
We define the following fundamental Brauer diagrams:
\begin{align*}
s_i^n &= \di{0.5\sca}{
\li[-\br][0]
\li[\br][0] 
\draw (0,0.5\h) node[]{$\dots$}; 
\draw (2*\br,-0.5\h) node[]{$i$};
\draw (4*\br,-0.5\h) node[]{$i+1$};
\li[4*\br][0]
\draw (5*\br,0.5\h) node[]{$\dots$};
\ocr[2*\br][0] \li[6*\br][0]}  &   \text{an $(n,n)$-Brauer diagram,}
\\
a_i^n &= \di{0.5\sca}{\li[-\br][0]
\li[\br][0] 
\draw (0,0.5\h) node[]{$\dots$};
\draw (2*\br,-0.5\h) node[]{$i$};
\draw (4*\br,-0.5\h) node[]{$i+1$};
\lin[4*\br][0][2*\br][\h]
\lin[5*\br][0][3*\br][\h]
\draw (5*\br,0.5\h) node[]{$\dots$};
\ca[2*\br][0] 
\lin[7*\br][0][5*\br][\h]} &  \text{an $(n+2,n)$-Brauer diagram,}
\\
u_i^n &= \di{0.5\sca}{\li[-\br][0]
\li[\br][0] 
\draw (0,0.5\h) node[]{$\dots$};
\draw (2*\br,1.5\h) node[]{$i$};
\draw (4*\br,1.5\h) node[]{$i+1$};
\draw (5*\br,0.5\h) node[]{$\dots$};
\lin[2*\br][0][4*\br][\h]
\lin[3*\br][0][5*\br][\h]
\cu[2*\br][\h]
 \lin[5*\br][0][7*\br][\h]} &  \text{an $(n,n+2)$-Brauer diagram.}
\end{align*}
\end{definition}
We can put a fundamental diagram $x_1$ on top of another fundamental diagram  $x_2$ when the number of bottom dots of $x_1$ matches the number of top dots of $x_2$. Connecting these matching dots, we obtain a new Brauer diagram which we denote by $x_1x_2$. Note that $x_1x_2$ may contain a closed loop, necessitating a rule to remove this loop from the diagram to really again obtain a Brauer diagram. For now, we can ignore it.   
It is clear that  every Brauer diagram can be decomposed into fundamental diagrams by making repeated use of this construction. However, this procedure is highly non-unique, as Figure \ref{different decompositions} shows.

\begin{figure}[h]
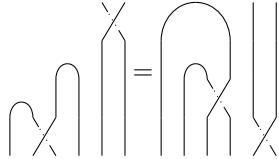

\begin{center}
\[
\di{0.5\sca}{\li \ocr[\br][0]\li[3*\br][0] \li[4*\br][0] \li[5*\br][0] 
\ca[0][\h] \li[2*\br][\h] \li[3*\br][\h] \li[4*\br][\h] \li[5*\br][\h] 
\ca[2*\br][2*\h] \li[4*\br][2*\h] \li[5*\br][2*\h]
\ocr[4*\br][3*\h] }=
\di{0.5\sca}{
\li \li[\br][0] \li[2*\br][0] \li[3*\br][0] \ocr[4*\br][0]
\li[0][\h] \li[\br][\h] \ocr[2*\br][\h] \li[4*\br][\h] \li[5*\br][\h]
\li[0][2*\h] \ca[\br][2*\h] \li[3*\br][2*\h] \li[4*\br][2*\h] \li[5*\br][2*\h]
\li[4*\br][3*\h] \li[5*\br][3*\h]
\draw  (0,3*\h) to [out=90, in=180] +(3*\br/2,\h);
\draw  (3*\br, 3*\h) to [out=90, in=0] +(-3*\br/2,\h);         
    }
\]
\caption{Two different decompositions into fundamental diagrams of the same Brauer diagram: $s_1^2 a_1^2 a_1^4s_2^6=a_1^2a_2^4s_3^6s_5^6$.}\label{different decompositions}
\end{center}
\end{figure}

We will choose for every Brauer diagram a distinguished decomposition into these fundamental diagrams, which we will call the standard expression. In such a standard expression all cups will be above the caps. We also want the left strand of a cup or a cap to be a straight line not encountering any crossings. Furthermore, we will  
order the height of the occurring caps and cups by the position of this left strand. 
Cups will be ascending from left to right, while caps will be descending. 

Let us now introduce standardly ordered cups and caps and a distinguished basis for the symmetric group algebra to define this standard expression.  	

Define  $I_s$ as a cup crossing $s$ propagating lines and  where the left strand of the arc is a straight line:
 \[ I_s \coloneqq  \di{0.5\sca,-0.5\sca}{\ca \uli[2*\br][0] \ucr[\br][-\h] \li[0][-\h] 
\draw[dotted] (2.5*\br, -0.5\h) to (3.5*\br,-0.5\h); 
\draw[dotted] (1.5*\br, -3.5*\h) to (2.5*\br,-3.5\h);
\draw[dotted] (2*\br, -\h) to (3*\br,-2*\h);  
 	 \uli[5*\br][0] \li[5*\br][-\h] \li[4*\br][-\h] \uli[4*\br][0]
\li[0][-2*\h] 	 \li[0][-3*\h] \li[0][-4*\h]
\li[\br][-2*\h] \li[\br][-3*\h] \li[\br][-4*\h]
 \li[3*\br][-4*\h] \ucr[3*\br][-3*\h]  \ucr[4*\br][-4*\h]
\li[4*\br][-2*\h]  
\li[5*\br][-2*\h] \li[5*\br][-3*\h]
 	 }  
. 
\]
We can embed such a cup $I_s$ into an $(n,n+2)$-Brauer diagram by adding straight propagating lines to the left and right
\begin{align}
 \label{elementary cup}
I_{s}^{n,a} \coloneqq \mathbb{1}^{a-1} \otimes I_s \otimes \mathbb{1}^{n-a-s+1}.
\end{align}
We call $I_{s}^{n,a}$ an elementary cup. Note that the left strand of the cup of $I_{s}^{n,a}$ is on the $a$th node. 
Then we can define the set of standardly ordered cups $I(r,n)$, containing  $(r,n)$-Brauer diagrams with only cups:
\[I(r,n)\coloneqq \left\{I_{s_1}^{n-2,a_1}  I_{s_2}^{n-4,a_2} \dots I_{s_{\frac{n-r}{2}}}^{r,a_{\frac{n-r}{2}}} \mid 
\begin{array}{l}
0 \leq s_i \leq n-2i, \\ a_i > a_j \text{ if } i<j  
\end{array} 
\right\}.
\] This set consists of compositions of elementary cups which are ordered such that the left strand of a cup is to the left of all left strands of cups above it.  
Note that this assures that the left strand of a cup is always a straight line to the dot above it since crossings only occur on right strands. 
\begin{figure}[h]
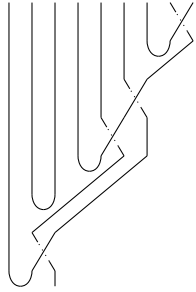

\begin{center}
\[
\di{0.5\sca}{ \ocr[7*\br][7*\h] \cu[6*\br][7*\h] \li[6*\br][7*\h] 
\lin[5*\br][8*\h][5*\br][6*\h] \ocr[5*\br][5*\h] \lin[6*\br][6*\h][8*\br][7*\h]
\lin[4*\br][8*\h][4*\br][5*\h] \ocr[4*\br][4*\h] 
\lin[3*\br][8*\h][3*\br][4*\h] \cu[3*\br][4*\h]
\lin[2*\br][8*\h][2*\br][3*\h]
\lin[\br][8*\h][\br][3*\h] \cu[\br][3*\h]
\lin[0][8*\h][0][\h] 
\cu[0][\h]
\ocr[\br][\h]
\dli[2*\br][\h]
\lin[6*\br][5*\h][6*\br][4*\h] 
\lin[\br][2*\h][5*\br][4*\h]
\lin[2*\br][2*\h][6*\br][4*\h]
    }
\]
\caption{Example of the diagram $I_1^{7,7}I_2^{5,4} I_0^{3,2}  I_1^{1,1}$ in $I(1,9)$.}
\end{center}
\end{figure}

Similarly, we can define standardly ordered caps. Set
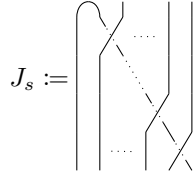
\[ J_s \coloneqq \di{0.5\sca}{\ca \uli[2*\br][0] \ocr[\br][-\h] \li[0][-\h] 
\draw[dotted] (2.5*\br, -0.5\h) to (3.5*\br,-0.5\h); 
\draw[dotted] (1.5*\br, -3.5*\h) to (2.5*\br,-3.5\h);
\draw[dotted] (2*\br, -\h) to (3*\br,-2*\h);  
 	 \uli[5*\br][0] \li[5*\br][-\h] \li[4*\br][-\h] \uli[4*\br][0]
\li[0][-2*\h] 	 \li[0][-3*\h] \li[0][-4*\h]
\li[\br][-2*\h] \li[\br][-3*\h] \li[\br][-4*\h]
 \li[3*\br][-4*\h] \ocr[3*\br][-3*\h]  \ocr[4*\br][-4*\h]
\li[4*\br][-2*\h]  
\li[5*\br][-2*\h] \li[5*\br][-3*\h]
 	 } 
 \]
and define the elementary cap by
\[
J_{s}^{n,a} \coloneqq \mathbb{1}^{a-1} \otimes J_s \otimes \mathbb{1}^{n-a-s+1}.
\]
This is a $(n+2,n)$-Brauer diagram where the left strand of the cap is on the $a$th node. 
We now look at combinations of elementary caps such that the left strand of a cap is to the left of all left strands of caps under it:
\[J(n,r)\coloneqq \left\{J_{s_1}^{r,a_1}  J_{s_2}^{r+2,a_2} \dots J_{s_{\frac{n-r}{2}}}^{n-2,a_{\frac{n-r}{2}}} \mid 
\begin{array}{l}
0 \leq s_i \leq n-2i,\\
 a_i < a_j \text{ if } i<j  
 \end{array}
 \right\}.
\]
Consider a basis $H(r)$ for the symmetric group algebra $\mK S_r$  consisting of  reduced expressions in the generators $s_1, \dots, s_{r-1}$. We will also interpret an element in $H(r)$  as the diagram obtained via the composition of the fundamental diagrams $s_i$ corresponding to this expression. The specific choice for the expressions does not matter to us. We only require that we have reduced expressions in the generators $s_i$ and that the subexpression $s_i s_{i+1}s_i$ does not occur. This is always possible since we have the braid relation $s_is_{i+1}s_i =s_{i+1}s_is_{i+1}$. 

\begin{proposition}\label{Proposition standard expression}
Every $(m,n)$-Brauer diagram has a unique decomposition into fundamental diagrams of the form $UXA$ where $U \in I(r,n)$, 
$X \in H(r)$ and $A \in J(m,r)$. 
\end{proposition}
We call this decomposition the standard expression. 
\begin{proof}
An $(m,n)$-Brauer diagram is a partitioning of $m+n$ dots into disjunct pairs where we have $m$ bottom dots and $n$ top dots.  
A pair which connects two top dots is called a cup and we order them as follows. We say that a cup is lower in the order than another cup if the left dot of the first cup is to the left of the left dot of the second cup. This ordering gives us a unique corresponding element $U$ in $I(n-2s,n)$, where $s$ is the number of cups in the Brauer diagram.  

Similarly, we call a pair that connects two bottom dots a cap and order them as follows. We say that a cap is higher in the order than another cap if the left dot of the first cap is to the left of the left dot of the second cap. This leads to a unique corresponding element $A$ in $J(m,m-2t)$, where $t$ is the number of caps in the Brauer diagram. 

Since the other pairs in the Brauer diagram correspond to pairings of a top and bottom dot, we necessarily have $m-2t = n-2s$. Moreover, these other pairs, corresponding to propagating lines, give a unique permutation of $r\coloneqq m-2t$ elements. Let $X$ be the  element in $H(r)$ corresponding to this permutation. We conclude that every Brauer diagram can be uniquely decomposed into fundamental diagrams such that it is of the form $UXA$ with $U \in I(r,n)$, 
$X \in H(r)$ and $A \in J(m,r)$. 
\end{proof}

\section{Monoidal supercategories}\label{Section monoidal supercategories}

In this section, we will recall monoidal supercategories as introduced by Brundan and Ellis in \cite{BrundanEllis}. A more detailed exposition can be found therein. 

\subsection{Definition of a monoidal supercategory}
A super vector space $V$ is a vector space over $\mK$ with a $\mathbb{Z}/2\mathbb{Z}$-grading, i.e. $V= V_{\oa} \oplus V_{\ob}$. Elements in $V_{\oa}$ are called even, and elements in $V_{\ob}$ are called odd. Together the even and odd elements give the homogeneous elements. For a homogeneous element, we set the parity $\abs{x} = i$ if $ x\in V_i$.

Let $\mathbf{svec}_{\mK}$ be the category of super vector spaces with grading preserving homomorphisms. 
\begin{definition}
A supercategory $\cC$ is defined as a category enriched over $\mathbf{svec}_{\mK}$.
\end{definition}
This means that for all $a,b \in \cC$ we have that $\Hom_{\cC}(a,b)$ 
is a vector space which decomposes as $\Hom_{\cC}(a,b)_{\oa}\oplus \Hom_{\cC}(a,b)_{\ob} $ and that the composition of morphisms is grading-preserving and linear.
 Thus $f\circ g \in \Hom_{\cC}(a,c)_{\abs{f} + \abs{ g}}$ for $f \in \Hom_{\cC}(b,c)_{\abs{f}}$ 	and $ g \in   \Hom_{\cC}(a,b)_{\abs{g}}$.
\begin{definition}
A superfunctor between two supercategories $\cC$ and $\cD$ is a functor $F\colon \cC \to \cD$ such that the map $\Hom_{\cC}(\lambda, \mu) \to \Hom_{\cD}(F(\lambda),F(\mu))$ is linear and even. 
\end{definition}
Let $\cC$ and $\cD$ be supercategories, then $\cC \boxtimes \cD$ is defined as the 
category which has as objects pairs of objects of $\cC$ and $\cD$ and morphisms are 
defined by 
\[ \Hom_{\cC \boxtimes \cD}((\lambda, \mu), (\sigma, \tau)) \coloneqq  \Hom_{\cC}(\lambda, \sigma) \otimes \Hom_{\cD}( \mu, \tau).   \]
Composition of morphisms is defined using the super interchange law:
\[
(f \otimes g) \circ (h  \otimes k) = (-1)^{\abs{h}\abs{g}}(f \circ h) \otimes (g \circ k).
\]
\begin{definition}[Strict monoidal supercategory]
A strict monoidal supercategory is a supercategory $\mathcal{C}$  with a superfunctor
\[ - \otimes - \colon \mathcal{C} \boxtimes \mathcal{C} \to \mathcal{C},
\]
 and a unit object $\mathbf{1}_{\mathcal{C}}$, such that we have 
 \[
 \mathbf{1}_{\mathcal{C}} \otimes -  = \id = - \otimes \mathbf{1}_{\mathcal{C}}, \qquad (-\otimes -) \otimes - = -\otimes (-\otimes -).
 \]
\end{definition}
An important difference between monoidal supercategories and ordinary monoidal categories is in the way composition of morphisms and the monoidal product interact with each other. In a monoidal supercategory, we have the so-called super interchange law
\[
(f \otimes g) \circ (h  \otimes k) = (-1)^{\abs{h}\abs{g}}(f \circ h) \otimes (g \circ k).
\]
\begin{definition}[Monoidal superfunctors]
A strict monoidal superfunctor $F$ between two strict monoidal supercategories $\mathcal{C}$ and $\mathcal{D}$ is a superfunctor $F\colon \cC \to \cD$ such that 
$F(a\otimes b) = F(a)\otimes F(b)$ and $F(\mathbf{1}_\cC) = \mathbf{1}_\cD$. 
\end{definition}
\begin{definition}[Monoidal opposite]
Let $(\cC, \otimes)$ be a strict monoidal supercategory. The opposite category $(\cC^{\op}, \otimes^{\op})$ is defined such that $\cC^{op}=\cC$ as categories, but the tensor product satisfies $f \otimes^{op} g \coloneqq (-1)^{\abs{f}\abs{g}} g \otimes f$.
\end{definition}
Note that this is a different notion than the dual of a category, which is defined by reversing all the arrows in a category, i.e. changing the source and the target for each morphism. 

We can depict morphisms graphically as follows.
A morphism $f \in \Hom_\mathcal{C} (\lambda, \mu)$ corresponds to the picture
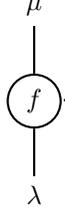
\[
\begin{tikzpicture}[baseline={(current bounding box.center)},scale=0.5,thick,>=angle 90]
 \begin{scope} 
\draw (0,0) node[circle,draw] (A) {$f$}; 
\draw (A) to (0,2);
\draw (A) to (0,-2);
\draw (0,-2.5) node {$\lambda$}; 
\draw (0,2.5) node {$\mu$}; 
 \end{scope} 
\end{tikzpicture}.
\] 
Composition is represented by putting one morphism on top of the other, while the monoidal product corresponds to putting morphisms next to each other:
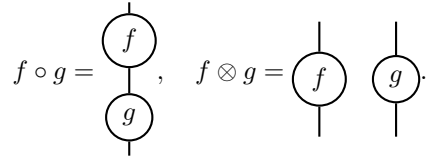
\[
f\circ g = \begin{tikzpicture}[baseline={(current bounding box.center)},scale=0.51,thick,>=angle 90]
 \begin{scope} 
\draw (0,1) node[circle,draw] (A) {$f$}; 
\draw (0,-1) node[circle,draw] (B) {$g$}; 
\draw (A) to (B);
\draw (A) to (0,2);
\draw (B) to (0,-2);
 \end{scope} 
\end{tikzpicture},
\quad
f \otimes g=
\begin{tikzpicture}[baseline={(current bounding box.center)},scale=0.51,thick,>=angle 90]
 \begin{scope} 
\draw (0,1) node[circle,draw] (A) {$f$}; 
\draw (A) to (0,2.5);
\draw (A) to (0,-0.5);
\draw (2,1) node[circle,draw](B) {$g$}; 
\draw (B) to (2,2.5);
\draw (B) to (2,-0.5);
 \end{scope} 
\end{tikzpicture}.
\]

Graphically, the super interchange law can then be depicted as follows:
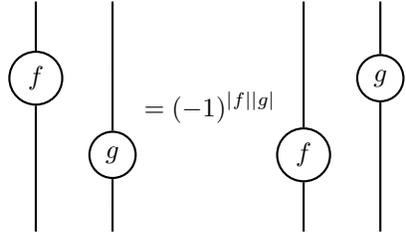
\[
\begin{tikzpicture}[baseline={(current bounding box.center)},scale=0.51,thick,>=angle 90]
 \begin{scope} 
\draw (0,2) node[circle,draw] (A) {$f$}; 
\draw (A) to (0,4);
\draw (A) to (0,-2);
\draw (2,0) node[circle,draw](B) {$g$}; 
\draw (B) to (2,4);
\draw (B) to (2,-2);
 \end{scope} 
\end{tikzpicture}
= (-1)^{\abs{f}\abs{g}}
\begin{tikzpicture}[baseline={(current bounding box.center)},scale=0.51,thick,>=angle 90]
 \begin{scope} 
\draw (2,2) node[circle,draw] (A) {$g$}; 
\draw (A) to (2,4);
\draw (A) to (2,-2);
\draw (0,0) node[circle,draw](B) {$f$}; 
\draw (B) to (0,4);
\draw (B) to (0,-2);
 \end{scope} 
\end{tikzpicture}.
\]
Note that we suppress identity morphisms.

\subsection{Example: the marked Brauer category} \label{section example Brauer category}
 The marked Brauer category $\mathcal{B}(\epsilon)$ introduced in \cite{KujawaTharp}   has as objects the natural numbers. The morphisms $\Hom_{\mathcal{B}(\epsilon)}(r,s)$ are given by linear combination of $(r,s)$-Brauer diagrams. We have two types of multiplication. \begin{itemize}
 \item Vertical multiplication corresponds to the composition of morphisms
\[\Hom_{\mathcal{B}(\epsilon)}(s,t) \times \Hom_{\mathcal{B}(\epsilon)}(r,s) \to \Hom_{\mathcal{B}(\epsilon)}(r,t) \]  defined via (periplectic) multiplication of Brauer diagrams. See \cite[Section 2]{KujawaTharp} for a precise definition. 
\item Horizontal multiplication is given on objects by $m\otimes n =m+n$ and on morphisms
\[\Hom_{\mathcal{B}(\epsilon)}(r,s) \otimes \Hom_{\mathcal{B}(\epsilon)}(r',s') \to \Hom_{\mathcal{B}(\epsilon)}(r+r',s+s') \] by putting Brauer diagrams next to each other.
\end{itemize}
As mentioned in \cite[Example 1.5 (iii)]{BrundanEllis} the marked Brauer category is a strict monoidal supercategory which can be presented graphically as follows. It has one generating object: $\cdot$ and three generating morphisms. 
One even generating morphism
$
\di{0.5\sca}{\cro}
$
and two generating morphisms 
$
\di{0.5\sca}{\ca}$ and $\di{0.5\sca}{\cu} 
$
of the same parity subject to the following relations
\begin{align*}
\di{\sca}{\cro \cro[0][\h]}
&=
\di{\sca}{\lili}
\; , \quad  
%
 \di{\sca}{\cro \li[2*\br][0] \li[0][\h] \cro[\br][\h] \cro[0][2*\h] \li[2*\br][2*\h]}
= 
\di{\sca}{\li \cro[\br][0] \cro[0][\h] \li[2*\br][\h] \li[0][2*\h] \cro[\br][2*\h]}
\; , \quad
\di{\sca}{\dli \ca \cu[\br][0] \uli[2*\br][0]}
=\;
\di{\sca}{\li}
\; ,\quad 
\di{\sca}{\uli \cu \ca[\br][0] \dli[2*\br][0]}={\color{red} \epsilon} \;
\di{\sca}{\li}
\; ,\quad
 \\
\di{\sca}{ \cu  \cro[\br][0] \li \dli[2*\br][0]}&=
\di{\sca}{\cro \dli \cu[\br][0] \li[2*\br][0]} \; , \quad
\di{\sca}{\cro \cu}= \;
\di{\sca}{\cu} \;.
\end{align*}
Here $\epsilon=1$ if $\di{0.5\sca}{\ca}$ and $\di{0.5\sca}{\cu}$ are even and $\epsilon=-1$ if they are odd. In the even case, we obtain a monoidal category, which corresponds to the Brauer category. For the odd case, we obtain a monoidal supercategory, which corresponds to the periplectic Brauer category.

\section{Motivating the definition of categories of Brauer type }
\label{Section motivating}
The goal of this paper is to obtain categories similar to the marked Brauer category introduced in the previous section (both the even and odd cases) by tweaking the relations. 
So we are interested in monoidal (super)categories generated by the same generating morphisms but with different relations. We also impose that the hom-spaces have bases given by Brauer diagrams and that the cross is an isomorphism satisfying the braid relation.

Consider a monoidal supercategory $\cC$ over $\mK$ with one generating object and three generating morphisms: the over-cross $H \in  \Hom_{\cC}(2,2)$ which is always even, and two morphisms of the same parity: the cap $A \in \Hom_{\cC} (2,0)$ and the cup $U \in \Hom_{\cC}(0,2)$.

Since we have one generating object, the objects of $\cC$ can be labelled by $\mathbb{N}$. An arbitrary morphism in $\cC$ is a combination of the three generating morphisms together with the identity morphism $I \in \Hom_{\cC}(1,1)$  using composition and tensoring. 

We will represent the morphisms $H, A,$ and $U$ graphically by 
 \[
 H=\begin{tikzpicture}[baseline={(current bounding box.center)},scale=1,thick,>=angle 90]
\begin{scope}
 \draw(0,0) to (0.15,0.25);
 \draw(0.15,0.25) to (0.45,0.75);
 \draw(0.45,0.75) to (0.6,1);
 \draw(0,1) to (0.15,0.75);
 \draw[dotted](0.15,0.75) to (0.45,0.25);
 \draw(0.45,0.25) to (0.6,0);
  \end{scope}
\end{tikzpicture},\quad 
A=\begin{tikzpicture}[baseline={(current bounding box.center)},scale=1,thick,>=angle 90]
\begin{scope}
\draw (0,0.3) to [out=90, in=180] +(0.3,0.4);
\draw (0.6,0.3) to [out=90, in=0] +(-0.3,0.4);
 \end{scope}
\end{tikzpicture},\quad 
U=\begin{tikzpicture}[baseline={(current bounding box.center)},scale=1,thick,>=angle 90]
\begin{scope}
\draw (0,0.6) to [out=-90, in=180] +(0.3,-0.4);
\draw (0.6,0.6) to [out=-90, in=0] +(-0.3,-0.4);
\end{scope}
\end{tikzpicture}.
 \]

We want that $\Hom_{\cC}(m,n)\cong \Hom_{\mathcal{B}(\epsilon)}(m,n)$ as vector spaces. This forces us to impose relations to simplify diagrams. 
For example, since
$\End_{\cC}(0)=\mK\; \id_0$, 
we need to impose
$
\di{0.5\sca}{\cu \ca}=
\delta,
$
for $\delta \in \mK$. 
Similarly, $\End_{\cC}(1)=\mK\; \id_1$ leads to
\begin{align}
\label{Straightening relation}
\di{0.5\sca}{\ca \cu[\br][0] \dli \uli[2*\br][0]}=
\parS \;\di{0.5\sca}{\li} \;, \quad
%
\di{0.5\sca}{\cu \ocr[\br][0] \li \ca[0][\h] \dli[2*\br][0] \uli[2*\br][\h]}
=q_1\;\di{0.5\sca}{\li}  \; ,
%
\quad
\di{0.5\sca}{\ocr \dli \uli[0][\h] \ca[\br][\h] \cu[\br][0] \li[2*\br][0]}
= q_2\; \di{0.5\sca}{\li} \;,
\end{align}
and $\Hom_{\cC}(2,0)= \mK \; \di{0.5\sca}{\ca}$, $\Hom_{\cC}(0,2)= \mK \;\di{0.5\sca}{\cu}$ forces relations of the following form
\begin{align*}
\di{0.5\sca}{\cu \ocr}= \lambda \;
  \di{0.5\sca}{\cu}\;,
\quad
\di{0.5\sca}{\ca[0][\h] \ocr}= \lambda' \; \di{0.5\sca}{\ca}\;.
\end{align*}

 A priori these constraints on the hom-spaces force us to introduce a whole plethora of relations and parameters. Remarkably, we can get by with at most  4 independent parameters and we get a fairly small list of possible categories as we will show in Section \ref{Section classes of categories}. 
 
Let us give an example of how we can reduce parameters.
The fact that $\End_{\cC}(1)=\mK\, \id_1$ also forces us to impose a relation
\begin{align}\label{Relation straightening pm}
\di{0.5\sca}{\cu \uli \ca[\br][0] \dli[2*\br][0]}=
\parS'
 \; \di{0.5\sca}{\li}\;.
\end{align} 
However, from the superinterchange law and relation \eqref{Straightening relation}, we can deduce
\begin{align*}
\parS' \di{0.5\sca}{\ca} = \di{0.5\sca}{\ca \li[0][-\h]  \dli[0][-\h] \li[\br][-\h] \cu[\br][-\h] \ca[2*\br][-\h] \dli[3*\br][-\h] } 
=\epsilon  
\di{0.5\sca}{\ca \cu[\br][0] \dli \li[2*\br][0] \ca[2*\br][\h] \li[3*\br][0] \dli[3*\br][0]}
=\epsilon \parS \di{0.5\sca}{\ca}.
\end{align*}
Here $\epsilon = -1$ if the cup and cap are odd and $\epsilon = 1$ if they are even.
So we obtain $\parS' = \epsilon \parS $ and we see that Relation \eqref{Relation straightening pm} did not introduce an extra parameter. 

We will call a monoidal supercategory a category of Brauer type if it is of the form we have described in this section.

\begin{definition}\label{Definition category of Brauer type}
A monoidal supercategory $\mathcal{C}$ is called of Brauer type if
\begin{enumerate}
\item The objects of $\mathcal{C}$ are generated by one object
\item The morphisms of $\mathcal{C}$ are generated by three morphisms :
\begin{itemize}
\item the over-cross $H \in  \Hom_{\cC}(2,2)$ denoted by $\di{0.5\sca}{\ocr}$,
\item the cap $A \in \Hom_{\cC} (2,0)$ denoted by $\di{0.5\sca}{\ca}$,
\item  the cup $U \in \Hom_{\cC}(0,2)$ denoted by $\di{0.5\sca}{\cu}$.
\end{itemize}
The morphism $H$ is always even, while $A$ and $U$ are either both even or both odd. 
\item The over-cross is an isomorphism that satisfies the braid relation.
\item The set of standard expressions for $(m,n)$-Brauer diagrams (as defined in Section \ref{Section Brauer diagrams}) form a basis for $\Hom_{\cC}(m,n)$.
\end{enumerate}
\end{definition}
Note that the marked Brauer category of Section \ref{section example Brauer category} satisfies this definition. We will see that also its deformations are categories of Brauer type. 

\section{The defining relations of the category} \label{Section relations}
Consider a category $\cC$ as in definition \ref{Definition category of Brauer type}. Then $\cC$ is a monoidal supercategory  which is generated by one object $\bullet$, an even morphism $\di{0.5\sca}{\ocr}$ and two morphisms $\di{0.5\sca}{\cu}$ and $\di{0.5\sca}{\ca}$ of the same parity.  
A morphism in $\cC$ is then a diagram obtained by combining the identity morphisms and these generating morphisms using composition and tensor products. Note that any morphism can be seen as the composition of the fundamental diagrams introduced in Definition \ref{Def fundamental diagrams}. 
We set $\epsilon =-1$ if the cup and cap are odd and $\epsilon =1$ if they are even. Note that in the last case, there is no super component and $\cC$ will just be a monoidal category. 

In this section, we introduce relations, which we call the defining relations of $\cC$. The existence of such relations in $\cC$ follows by definition from the fact that the set of standard expressions of $(m,n)$-Brauer diagrams form a basis for $\Hom_{\cC}(m,n)$. 
For example, $\{\di{0.5\sca}{\culi}, \di{0.5\sca}{\ccr}, \di{0.5\sca}{\licu} \}$ is a basis for $\Hom_{\cC}(1,3)$ and thus every $(1,3)$-Brauer diagram can be expressed as a linear combination of these three diagrams. 

The defining relations of $\cC$ are the following:
\begin{itemize}
\item Untwisting and upside-down untwisting
\begin{align*}
\di{\sca}{
\ocr 
\cu
}
 = \lambda  \;\di{\sca}{
 \cu
}\;,
\quad 
\di{\sca}{
\ocr \ca[0][\h]
}= \lambda' \; \di{\sca}{\ca}
\end{align*}
\item Looping
\begin{align*}
\di{\sca}{
\cu \ca
}=\delta\;
\end{align*}

\item  Straightening  and upside-down straightening
\begin{align*}
\di{\sca}{
\dli \ca \cu[\br][0] \uli[2*\br][0] 
}=\parS \; \di{\sca}{\li}
\;, \quad 
\di{\sca}{
\uli \cu \ca[\br][0] \dli[2*\br][0] 
}=\parS '\; \di{\sca}{\li}
\end{align*}

\item Delooping
\begin{align*}
\di{\sca}{
\cu \ocr[\br][0] \li \ca[0][\h] \dli[2*\br][0] \uli[2*\br][\h] 
}= \parq  \; \di{\sca}{\li} 
\end{align*}

\item Twisting
\begin{align*}
\di{\sca}{
\ocr \ocr[0][\h]
} = a\; \di{\sca}{\li \li[\br][0] } + b \; \di{\sca}{\ocr} +c\; \di{\sca}{\cc}
\end{align*}

\item Sliding
\begin{align*}
\di{\sca}{ \ocr \dli \cu[\br][0] \li[2*\br][0] } =d\; \di{\sca}{\culi}  + e\; \di{\sca}{\ccr} +f\;  \di{\sca}{\licu }
\end{align*}

\item Pulling

\begin{align*}
\di{\sca}{\dli[2*\br][0] \cu \ocr[\br][0] \li \li[2*\br][\h] \ocr[0][\h] }= D\;\di{\sca}{\culi}  + E\; \di{\sca}{\ccr} +F\;  \di{\sca}{\licu }
\end{align*}

\item Upside-down sliding:
\begin{align*}
\di{\sca}{\crc} =d'\; \di{\sca}{\cali} + e'\; \di{\sca}{ \li \ocr[\br][0] \ca[0][\h] \dli} +f'\; \di{\sca}{\lica} 
\end{align*}

\item Upside-down Pulling

\begin{align*}
\di{\sca}{ \ca \ocr[\br][-\h] \li[0][-\h] \li[2*\br][-2*\h] \ocr[0][-2*\h] \uli[2*\br][0] }= D'\; \di{\sca}{\cali} + E'\; \di{\sca}{ \li \ocr[\br][0] \ca[0][\h] \dli} +F'\; \di{\sca}{\lica}  
\end{align*}

\item Braiding

\begin{align*}
\di{\sca}{\ocr \li[2*\br][0] \li[0][\h] \ocr[\br][\h] \ocr[0][2*\h] \li[2*\br][2*\h]} = \di{\sca}{\ocr[\br][0] \li[0][0] \li[2*\br][\h] \ocr[0][\h] \ocr[\br][2*\h] \li[0][2*\h]}\; .
\end{align*}
\end{itemize}
Note that we have introduced the following parameters: \[\lambda,\lambda',\parS ,\parS ',\delta,\parq ,a,b,c,d,e,f,d',e',f',D,E,F,D',E',F'.\]
 We will always simplify diagrams using these relations from left to right. Thus we will replace a local occurrence of a left-hand side diagram with the corresponding linear combination on the right-hand side. In this way, the simplifying procedure will always end and result in a linear combination of standard expressions of Brauer diagrams, as we will show in Theorem \ref{theorem simplifying}.  
However, to have a well-defined category, we  also want this simplification to be unique. So, if different relations can be used to simplify a diagram, they should in the end lead to the same linear combination of diagrams. Demanding this will give us equations that the parameters should satisfy. This will be addressed in Section \ref{Section rewritable diagrams}.

\begin{theorem}\label{theorem simplifying}
Consider a diagram composed of an arbitrary number of fundamental diagrams in $\cC$. 
Simplifying this diagram using the defining relations of $\cC$ will always end in a linear combination of standard expressions of Brauer diagrams.
\end{theorem}
\begin{proof}
Note first that the standard expression of a diagram can not be simplified using the relations. This follows since in the standard expression cups are always above caps and the left strand of a cup or a cap is always a straight line. Furthermore, we choose a basis of the symmetric group algebra in such a way that the expression is reduced and the relation $s_is_{i+1}s_i$ does not occur. 
	
We will use induction on the number of fundamental diagrams. Note that a fundamental diagram is already a standard expression, covering the induction base case. 
So assume that we have a diagram $X$ consisting of $k+1$ fundamental diagrams. Then $X= x_0 X'$, where $x_0$ is a fundamental diagram and $X'$ consists of $k$ fundamental diagrams. By the induction hypothesis, we have $X'=\sum \lambda_i x_i$, where the $x_i$ are standard expressions consisting of at most $k$ fundamental diagrams. 

We will now show that $x_0x_i$ is a linear combination of standard expressions consisting of at most $k+1$ fundamental diagrams. We will  consider the cases where $x_0$ is a cup, cap or over-crossing separately.
\begin{itemize}
\item 
If $x_0$ is a fundamental cup, then using the super interchange rule, we can pull the cup down to the appropriate level in the ordering, making $x_0x_i$ into a standard expression. This is always possible without crossing lines since all left-side strands of cups are straight lines in the standard expression $x_i$.

\item
Assume now that $x_0$ is a fundamental cap of the form $a_r^n$ and $x_i$ is of the form $I_s^{n,a'}x_i'$, where $I_s^{n,a'}$ is an elementary cup as in Equation \eqref{elementary cup}. 
If $r+1<a'$ or $r>s+a'-1$, i.e. if the right strand of the cap is to the left of the  cup or the left strand of the cap is to the right of the  cup, then  we can use the superinterchange rule to switch the level of the elementary cup and the fundamental cap. This leads to $x_0x_i=\pm I_s^{n,a'}x_0 x_i'$ and for $x_0x_i'$ we can use the induction hypothesis to obtain a linear combination of standard expressions. Note that the cup of $I_s^{n,a'}$ is to the right of the other cups in $x_i'$ since $x_i  = I_s^{n,a'}x_i'$ is a standard expression. If the cup of $I_s^{n,a'}$ is still to the right of the cups in $ x_0 x_i'$, then $I_s^{n,a'}x_0 x_i'$ is a standard expression. The only simplifying relations that could introduce a cup in $x_0 x_i'$ to the right of $I_s^{n,a'}$ are twisting and the right-most term in pulling. But these relations reduce the number of fundamental diagrams, so we can use the induction hypothesis to rewrite $I_s^{n,a'}x_0 x_i'$ into a standard expression.

\begin{figure}[h]
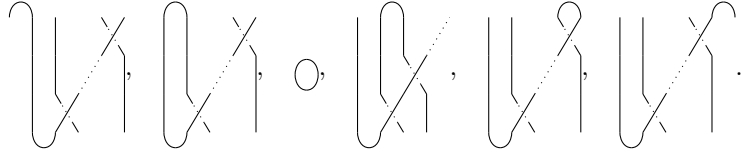

\begin{center}
\[
\di{0.5\sca}{\cu \li \li[0][\h] \li[0][2*\h] \li[\br][\h] \li[\br][2*\h] 
\draw[dotted] (2*\br, \h) to (3*\br,2*\h);  
\ocr[\br][0] \ocr[3*\br][2*\h] \li[4*\br][\h] \li[4*\br][0]   \ca[-\br][3*\h]  }, 
\quad 
\di{0.5\sca}{\cu \li \li[0][\h] \li[0][2*\h] \li[\br][\h] \li[\br][2*\h] 
\draw[dotted] (2*\br, \h) to (3*\br,2*\h);  
\ocr[\br][0] \ocr[3*\br][2*\h] \li[4*\br][\h] \li[4*\br][0]   \ca[0][3*\h] },
\quad 
\di{0.5\sca}{\cu \ca},
\quad
\di{0.5\sca}{\cu \li \li[0][\h] \li[0][2*\h] \li[\br][\h] \li[\br][2*\h] 
\draw[dotted] (3*\br, 2*\h) to (4*\br,3*\h); 
\li[2*\br][2*\h] 
\ocr[\br][0] \ocr[2*\br][\h]  \li[3*\br][0]   \ca[\br][3*\h] },
\quad
\di{0.5\sca}{\cu \li \li[0][\h] \li[0][2*\h] \li[\br][\h] \li[\br][2*\h] 
\draw[dotted] (2*\br, \h) to (3*\br,2*\h);  
\ocr[\br][0] \ocr[3*\br][2*\h] \li[4*\br][\h] \li[4*\br][0]  \ca[3*\br][3*\h] },
\quad
\di{0.5\sca}{\cu \li \li[0][\h] \li[0][2*\h] \li[\br][\h] \li[\br][2*\h] 
\draw[dotted] (2*\br, \h) to (3*\br,2*\h);  
\ocr[\br][0] \ocr[3*\br][2*\h] \li[4*\br][\h] \li[4*\br][0]   \ca[4*\br][3*\h] }.
\]
\caption{The different possible situations where the fundamental cap and cup overlap.\label{figure overlapping cup and cap}}
\end{center}
\end{figure}

If the fundamental cap and the cup overlap, see Figure \ref{figure overlapping cup and cap}, we can simplify the resulting diagram using   straightening, delooping, looping, upside-down pulling, untwisting, or upside-down sliding. Except for the upside-down sliding, all these relations reduce the number of fundamental diagrams, and then we can use our induction hypothesis to obtain a linear combination of standard expressions for $x_0x_i$. The upside-down sliding relation $\di{0.5\sca}{\crc} =d'\; \di{0.5\sca}{\cali} + e'\; \di{0.5\sca}{ \li \ocr[\br][0] \ca[0][\h] \dli} +f'\; \di{0.5\sca}{\lica} 
$ is needed in the most right diagram in Figure \ref{figure overlapping cup and cap}.
We only have to consider the term in $\di{0.5\sca}{ \li \ocr[\br][0] \ca[0][\h] \dli}$ as for the others term the numbers of fundamental diagrams has again been reduced. Observe that we can then repeat upside-down sliding until we obtain $\di{0.5\sca}{\cu \li  \li[\br][0] \li[2*\br][0] \ca[\br][\h]}$. On this, we can apply straightening which reduces the number of fundamental diagrams and allows us to apply the induction hypothesis.  

Assume $x_0$ is still a fundamental cap of the form $a_r$ but the diagram $x_i$ does not contain any cups. Then we repeatedly apply upside-down sliding and upside-down pulling until the left strand of the cup is a straight line. We can then use the superinterchange rule to bring this cup to the appropriate level to obtain a standard expression.   

\item 
We will consider now the case when $x_0$ is a fundamental cross $s_r$ and $x_i$ is of the form $I_s^{n,a'}x_i'$. If the strands of the cross $s_r$ are to the left or the right of the strands of the cup we can use the superinterchange rule to switch the cross and the elementary cup, while if the strands of $s_r$ are between the strands of the cup, we can use the braid rule to switch the cross and the cup. In these cases we have $x_0 I_s^{n,a'}x_i' = \pm I_s^{n,a'}x_0x_i'$ and on $x_0x_i'$ we can apply the induction hypothesis. This leaves us the four cases where the strands of the cross overlap with the strands of the cup.
If $r=a'-1$ then we can apply sliding $\di{0.5\sca}{ \ocr \dli \cu[\br][0] \li[2*\br][0] } =d\; \di{0.5\sca}{\culi}  + e\; \di{0.5\sca}{\ccr } +f\;  \di{0.5\sca}{\licu }
$. The terms in $ \di{0.5\sca}{\culi}$ and $\di{0.5\sca}{\licu }$ reduce the number of fundamental diagrams, while the term in $\di{0.5\sca}{\ccr }$ leads to $I_{s+1}^{n,a'-1}$. Hence $x_0 I_s^{n,a'} x_i'= I_{s+1}^{n,a'-1}x_i' +$ terms with less fundamental diagrams. 
If $r=a'$, we can apply pulling to reduce the number of fundamental diagrams.
while for $r=a'+s-2$ we can use twisting to reduce the number of fundamental diagrams. 
If $r=a'+s-1$, we immediately have that $x_0 I_s^{n,a'}= I_{s+1}^{n,a'}$.

For the last case, $x_0$ is a fundamental cross $s_r$ and $x_i$ does not contain any cup. Then $x_i = X A$, where $X \in H(n)$. We can use the braid relation to rewrite $x_0 X$ into a linear combination $\sum_l \mu_l X_l$ in $H(n)$ where every occurrence of $s_js_j$ can be discarded since the twisting relation reduces the number of fundamental diagrams. This allows us to obtain $x_0 x_i = x_0 XA = \sum_l \mu_l X_l A $+ terms with less fundamental diagrams.   
\end{itemize}
Using induction, this proves the theorem.
\end{proof}
The previous theorem shows that the Brauer diagrams, which we represent using their standard expression, form a spanning set for the hom-spaces of the monoidal supercategory $\cC$ generated by the cup, cap and over-cross and satisfying the defining relations in this section. 

From now on, we will also always assume that the over-cross is invertible. This has the following implications. 

\begin{lemma}\label{Lemma invertible}
If the over-cross $\di{0.5\sca}{\ocr}$ is invertible, then the inverse is given by 
\begin{align*}
\di{\sca}{\ucr}= -\frac{b}{a}\; \di{\sca}{\li \li[\br][0]} + \frac{1}{a}\; \di{\sca}{\ocr}-\frac{c}{\lambda a}\; \di{\sca}{\cc}.
\end{align*}
Furthermore, $\lambda$, $\lambda'$ and $a$ are non-zero.
\end{lemma}

\begin{proof}
Denote the inverse of $\di{0.5\sca}{\ocr}$ by the under-cross $\di{0.5\sca}{\ucr}$.
First note 
\[
\di{0.5\sca}{\cu}= \di{0.5\sca}{\cu \ocr \ucr[0][\h]} = \lambda \di{0.5\sca}{\ucr \cu}.
\]
Hence $\lambda$ must be non-zero and similar also $\lambda'$ must be non-zero. 

Multiplying the twisting relation on the left by the inverse $\di{0.5\sca}{\ucr}$ we obtain
\[
\di{0.5\sca}{\ocr} = a\; \di{0.5\sca}{\ucr}+ b\; \di{0.5\sca}{\lili} + c \lambda^{-1}\; \di{0.5\sca}{\cc}.
\]
Since the over-cross is by assumption linearly independent from $\di{0.5\sca}{\lili}$ and $\di{0.5\sca}{\cc}$, the parameter $a$ can not be zero. 
Then rewriting the previous relation gives us the lemma.
\end{proof}

\section{Establishing the equations} \label{Section rewritable diagrams}
Consider the monoidal supercategory $\cC$ generated by the cup, cap and over-cross and satisfying the defining relations of Section \ref{Section relations} and assume furthermore that the over-cross is invertible. Recall that these relations depend on parameters $\lambda,\lambda',\parS ,\parS ',\delta,\parq ,a,b,c,d,e,f,d',e',f',D,E,F,D',E',F'$, where $a,\lambda,\lambda'$ are non-zero. We will now derive the constraints these parameters have to satisfy such that the relations do not lead to contradictions and such that if we simplify a diagram using different relations, the resulting linear combination of diagrams we obtain is the same. 

Take two defining relations and consider the diagrams occurring on the left side of these relations. 
Since simplifying is a local operation, we only have to look at diagrams which contain these two diagrams as subdiagrams in such a way that they overlap. If this happens, we will simplify the diagram using these two relations in two different ways. We will then derive the equations which express that the final result of these simplifications is unique. 

The possible diagrams we have to consider are the following. We coloured the part of the diagram that overlaps. 
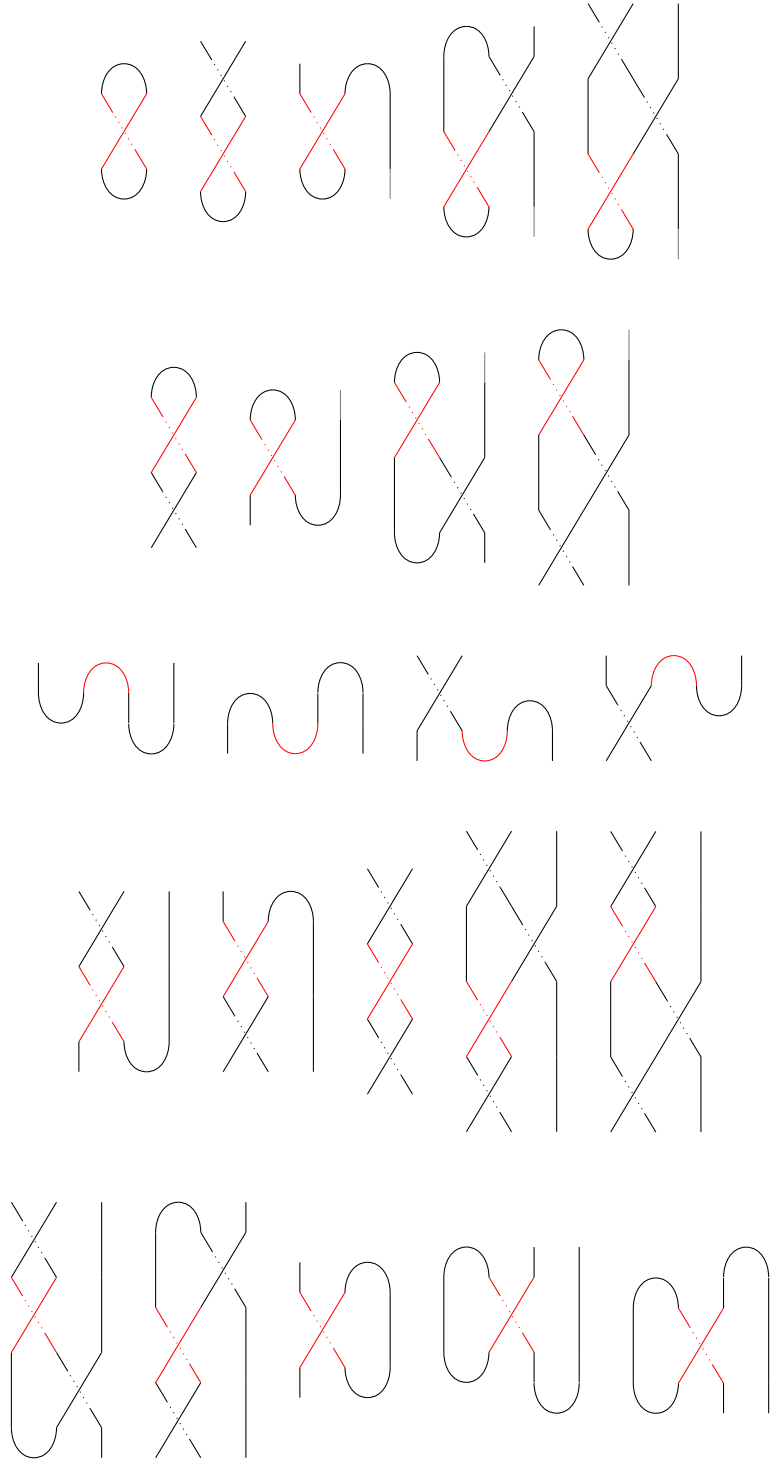
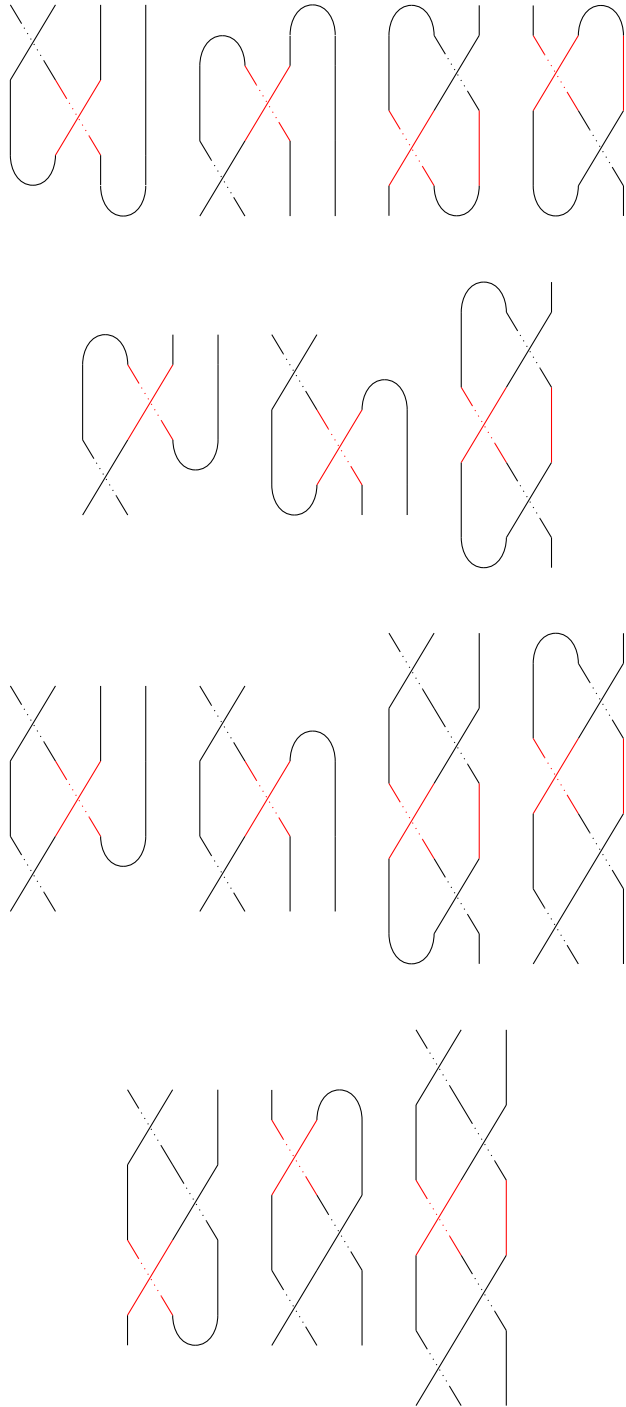
\begin{align*} 
\di{\sca}{\cu \ca[0][\h] \ocr[0][0][red]} \qquad
\di{\sca}{\cu \ocr[0][0][red] \ocr[0][\h]} \qquad
\di{\sca}{\cu \ocr[0][0][red] \ca[\br][\h] \uli[0][\h] \dli[2*\br][0][gray] \li[2*\br][0]} \qquad
\di{\sca}{\ocr[0][0][red] \cu \ca[0][2*\h] \ocr[\br][\h] \li[0][\h] \li[2*\br][0] \uli[2*\br][2*\h] \dli[2*\br][0][gray]} \qquad
\di{\sca}{\ocr[0][0][red] \cu \ocr[0][2*\h] \ocr[\br][\h] \li[0][\h] \li[2*\br][0] \li[2*\br][2*\h] \dli[2*\br][0][gray]}
\end{align*}

\begin{align*}
\di{-\sca,-\sca}{\cu \ocr[0][0][red] \ocr[0][\h]} \qquad
\di{\sca,-\sca}{\cu \ucr[0][0][red] \ca[\br][\h] \uli[0][\h] \dli[2*\br][0][gray] \li[2*\br][0]} \qquad
\di{\sca,-\sca}{\ucr[0][0][red] \cu \ca[0][2*\h] \ucr[\br][\h] \li[0][\h] \li[2*\br][0] \uli[2*\br][2*\h] \dli[2*\br][0][gray]} \qquad
\di{\sca,-\sca}{\ucr[0][0][red] \cu \ucr[0][2*\h] \ucr[\br][\h] \li[0][\h] \li[2*\br][0] \li[2*\br][2*\h] \dli[2*\br][0][gray]} 
\end{align*}

\begin{align*}
\di{\sca}{\cu \ca[\br][0][red] \uli \dli[2*\br][0] \cu[2*\br][-0.4\h] \uli[3*\br][0] \dli[3*\br][0] } \qquad 
\di{\sca,-\sca}{\cu \ca[\br][0][red] \uli \dli[2*\br][0] \cu[2*\br][-0.4\h] \uli[3*\br][0] \dli[3*\br][0] } \qquad 
\di{\sca}{\dli \ocr \cu[\br][0][red] \ca[2*\br][0] \dli[3*\br][0]} \qquad
\di{\sca,-\sca}{\dli \ucr \cu[\br][0][red] \ca[2*\br][0] \dli[3*\br][0]}
\end{align*}

\begin{align*} 
\di{\sca}{\dli \ocr[0][0][red] \ocr[0][\h] \cu[\br][0] \li[2*\br][0]\li[2*\br][\h] } \qquad
\di{\sca,-\sca}{\dli \ucr[0][0][red] \ucr[0][\h] \cu[\br][0] \li[2*\br][0]\li[2*\br][\h] } \qquad
\di{\sca}{\ocr \ocr[0][\h][red] \ocr[0][2*\h]} \qquad
\di{\sca}{\ocr \li[2*\br][0] \ocr[0][\h][red] \ocr[0][3*\h] \ocr[\br][2*\h] \li[2*\br][\h] \li[2*\br][3*\h] \li[0][2*\h] }\qquad
\di{\sca,-\sca}{\ucr \li[2*\br][0] \ucr[0][\h][red] \ucr[0][3*\h] \ucr[\br][2*\h] \li[2*\br][\h] \li[2*\br][3*\h] \li[0][2*\h] }
\end{align*}

\begin{align*}
\di{\sca}{\cu \li \ocr[\br][0] \ocr[0][\h][red] \ocr[0][2*\h]   \li[2*\br][\h] \li[2*\br][2*\h] \dli[2*\br][0] } \qquad
\di{\sca,-\sca}{\cu \li \ucr[\br][0] \ucr[0][\h][red] \ucr[0][2*\h]   \li[2*\br][\h] \li[2*\br][2*\h] \dli[2*\br][0] } \qquad
\di{\sca}{\ocr[0][0][red] \dli \cu[\br][0] \li[2*\br][0] \ca[\br][\h] \uli[0][\h]} \qquad
\di{\sca}{\cu \li \ca[0][\h] \ocr[\br][0][red] \uli[2*\br][\h] \dli[2*\br][0] \cu[2*\br][-0.4\h] \dli[3*\br][0] \li[3*\br][0] \uli[3*\br][\h]  }  \qquad
\di{\sca,-\sca}{\cu \li \ca[0][\h] \ucr[\br][0][red] \uli[2*\br][\h] \dli[2*\br][0] \cu[2*\br][-0.4\h] \dli[3*\br][0] \li[3*\br][0] \uli[3*\br][\h]  } 
\end{align*}

\begin{align*}
\di{\sca}{\cu \li \ocr[\br][0][red] \li[2*\br][\h] \ocr[0][\h] \dli[2*\br][0] \dli[3*\br][0] \cu[2*\br][-0.4\h] \li[3*\br][0] \li[3*\br][\h]} \qquad 
\di{\sca,-\sca}{\cu \li \ucr[\br][0][red] \li[2*\br][\h] \ucr[0][\h] \dli[2*\br][0] \dli[3*\br][0] \cu[2*\br][-0.4\h] \li[3*\br][0] \li[3*\br][\h]} \qquad
\di{\sca}{ \ocr[0][0][red] \ocr[\br][\h] \ca[0][2*\h] \li[0][\h] \uli[2*\br][2*\h] \li[2*\br][0][red] \cu[\br][0] \dli } \qquad
\di{\sca,-\sca}{ \ucr[0][0][red] \ucr[\br][\h] \ca[0][2*\h] \li[0][\h] \uli[2*\br][2*\h] \li[2*\br][0][red] \cu[\br][0] \dli }
\end{align*}

\begin{align*}
\di{\sca}{ \ocr[0][0] \ocr[\br][\h][red] \ca[0][2*\h] \li[0][\h] \uli[2*\br][2*\h] \li[3*\br][\h]  \uli[3*\br][2*\h] \cu[2*\br][\h]  } \qquad
\di{\sca,-\sca}{ \ucr[0][0] \ucr[\br][\h][red] \ca[0][2*\h] \li[0][\h] \uli[2*\br][2*\h] \li[3*\br][\h]  \uli[3*\br][2*\h] \cu[2*\br][\h]  } \qquad
\di{\sca}{\cu \dli[2*\br][0] \li \ocr[\br][0] \ocr[0][\h][red] \li[2*\br][\h][red] \li[0][2*\h] \ocr[\br][2*\h] \ca[0][3*\h] \uli[2*\br][3*\h]} 
\end{align*}

\begin{align*}
\di{\sca}{ \li[3*\br][2*\h] \li[3*\br][3*\h] \cu[2*\br][2*\h] \ocr[0][\h] \ocr[0][3*\h] \ocr[\br][2*\h][red]  \li[2*\br][3*\h] \li[0][2*\h] }\qquad
\di{\sca,-\sca}{ \li[3*\br][2*\h] \li[3*\br][3*\h] \cu[2*\br][2*\h] \ucr[0][\h] \ucr[0][3*\h] \ucr[\br][2*\h][red]  \li[2*\br][3*\h] \li[0][2*\h] }\qquad
\di{\sca}{\dli[2*\br][0] \ocr[\br][0] \li \cu[0][0] \ocr[0][\h][red] \ocr[0][3*\h] \ocr[\br][2*\h] \li[2*\br][\h][red]  \li[2*\br][3*\h] \li[0][2*\h] }\qquad
\di{\sca,-\sca}{\dli[2*\br][0] \ucr[\br][0] \li \cu[0][0] \ucr[0][\h][red] \ucr[0][3*\h] \ucr[\br][2*\h] \li[2*\br][\h][red]  \li[2*\br][3*\h] \li[0][2*\h] }
\end{align*}

\begin{align*}
\di{\sca}{\dli[0][\h] \cu[\br][\h] \ocr[0][\h][red] \ocr[0][3*\h] \ocr[\br][2*\h] \li[2*\br][\h] \li[2*\br][3*\h] \li[0][2*\h] }\qquad
\di{\sca,-\sca}{\dli[0][\h] \cu[\br][\h] \ucr[0][\h][red] \ucr[0][3*\h] \ucr[\br][2*\h] \li[2*\br][\h] \li[2*\br][3*\h] \li[0][2*\h] } \qquad
\di{\sca}{\ocr \li[2*\br][0] \li[0][\h] \ocr[\br][\h] \li[2*\br][2*\h][red] \ocr[0][2*\h][red] \li[0][3*\h] \ocr[\br][3*\h] \li[2*\br][4*\h] \ocr[0][4*\h]  }
\end{align*}

The calculations to obtain the equations expressing that rewriting the above overlapping diagrams gives a consistent result are straightforward but long. Therefore we put them in the appendix. We will now summarize here the results of Appendix \ref{Appendix rewrite equations}. 
The parameters $\parS ',d,d', D,E,F,D',E,F'$ can be expressed in the other parameters as follows
\begin{align*}
\parS '&=\epsilon\parS \qquad d=-ef' \qquad d'=-e'f \qquad E=b-f \qquad E'=b-f' \\
 D&= \frac{aE}{\lambda} \qquad D' = \frac{aE'}{\lambda'}  \qquad F= \frac{a}{e} \qquad F'=\frac{a}{e'}. 
\end{align*}
Furthermore
\begin{align}\label{Equation e,e'}
e^4=e'^4=1.
\end{align}
In particular $e$ and $e'$ are non-zero. 

If $\lambda=\lambda'$, we have the following equation
\begin{align}\label{Equation lambda=mu}
\lambda^2-b\lambda-c\delta&=a, 
\end{align}
while if $\lambda\not=\lambda'$, we have
\begin{align}\label{Equation mu not lambda}
c=\delta=0 \quad a=-\lambda'\lambda \quad b=\lambda'+\lambda. 	
\end{align}

From the equation
\begin{align}
f(b-f) &= 0 & f'(b-f') &= 0,
\end{align}
we conclude that we can distinguish four separate cases:
\begin{itemize}
\item $f=f'=0$,
\item $f=0$ and $f'=b\not=0$,
\item $f=b\not=0$ and $f'=0$,
\item $f=f'=b\not=0.$
\end{itemize}
We then have to solve the following equations for each case:
\begin{equation}\label{Equation f f'b}
\begin{aligned}
a(b-f)/\lambda &= \frac{c\parS }{e}-(b-\lambda-f)f' \\
a(b-f')/\lambda' &= \frac{c\parS '}{e'}-(b-\lambda'-f')f\\
\epsilon e^2(a(b-f)/\lambda -f\lambda) &= (f'^2-ff'-f'\lambda+f\lambda)+e c\parS' \\
\epsilon e'^2(a(b-f')/\lambda' -f'\lambda') &= (f^2-ff'-f\lambda'+f'\lambda')+e'c\parS \\
\epsilon e^2(b-2f) &= (b-2f') \\
\epsilon e'^2(b-2f') &= (b-2f), \\
f'(ec\parS '+f\lambda)&= a(b-f)(b-f')/\lambda \\
f(e'c\parS +f'\lambda')&= a(b-f)(b-f')/\lambda' \\
\lambda(b-f-f') -ec\parS '&=(b-f)(b-f') \\
\lambda'(b-f'-f) -e'c\parS &=(b-f)(b-f') \\
\lambda (ec\parS '+f\lambda-f^2) -ec\parS 'f &= a(b-f')  \\
\lambda' (e'c\parS +f'\lambda'-f'^2) -e'c\parS f' &= a(b-f). 
\end{aligned}
\end{equation}

The parameters should also satisfy the following equations
\begin{equation} \label{Equation rest}
\begin{aligned}
c\parq  &= a(b-f-f')  \\
\lambda (\parS e'c +ff')+c\delta f'&= \lambda' (\parS' ec+ff')+c\delta f= a(b-f-f') \\
\parq  & = \parS '(d+e\lambda)+f\delta= \parS (d'+e'\lambda') + f'\delta \\
\parS (\lambda'-\lambda)(1+\epsilon e^2)& =0 \\
(\lambda-b+f')\parq  &= \frac{a}{\lambda'}(b-f')\delta+a\parS 'e \\
 (\lambda'-b+f)\parq &= \frac{a}{\lambda}(b-f) \delta + a\parS e'\\
 (b-f)(\parq \lambda+a\delta ) + \parS a\lambda e'  &=(b-f')(\parq \lambda'+a\delta ) + \parS 'a\lambda' e  \\
a\delta(e-\frac{1}{e'}) &=(b-f')(\parS \frac{a}{\lambda'} -e\parq ) +(b-f'-f)\lambda\parS-ec\parS \parS '\\
a\delta(e'-\frac{1}{e}) &= (b-f)(\parS '\frac{a}{\lambda} -e'\parq )+(b-f -f')\lambda' \parS' -e'c\parS \parS' 
\end{aligned}
\end{equation}
and
\begin{equation} \label{Equation DEF}
\begin{aligned}
a(bE+c\parS' e )&= D^2+Ea\lambda + EbD+ Ec\parS' d +F\lambda d \\
a\lambda+bD+b^2E +c\parS' d +c\parS' eb &= DE +bE^2+Ec\parS' e + F\lambda e \\
b E c\parS' +bF\lambda + c^2 \parS '^2 e + c\parS' \lambda  f &= DF + EbF +Ec\parS' f + F \lambda f\\
a(bE'+c\parS e' )&= D'^2+E'a\lambda' + E'bD'+ E'c\parS d' +F'\lambda' d' \\
a\lambda'+bD'+b^2E' +c\parS d' +c\parS e'b &= D'E' +bE'^2+E'c\parS e' + F'\lambda' e' \\
b E' c\parS  +bF'\lambda' + c^2 \parS ^2 e' + c\parS \lambda' f' &= D'F' + E'bF' +E'c\parS f' + F' \lambda' f' .
\end{aligned}
\end{equation}

Furthermore, if $c$ is non-zero, we have the following extra equations
\begin{equation}\label{Equation c non-zero}
\begin{aligned}
f&=f' \\
e^2\epsilon f  & =e'^2\epsilon f=f \\
E&=E'=0  \\
ee'&=1 \\
ec\parS '+f\lambda&=0 & e'c\parS +f'\lambda'&=0 \\
\lambda &=\lambda', 
\end{aligned}
\end{equation}
while if $\parS $ is non-zero, we also have
\begin{align}\label{Equation S-non-zero}
ee'=1.
\end{align}

\begin{theorem}\label{Unique simplification}
Assume we have parameters $\lambda, \lambda', \parS ,\delta,\parq ,a, b,c,e,e',f,f'$ which satisfy Equations \eqref{Equation e,e'}  to \eqref{Equation S-non-zero}. Then the category $\mathcal{C}$ monoidally generated by the over-cross, the cup and the cap satisfying the defining relation of Section \ref{Section relations}  is well-defined. 
\end{theorem}
\begin{proof}
The way we established the equations immediately implies that different simplifications lead to the same result. Thus we can not have any contradictions in the relations. 
\end{proof}

\section{Categories of Brauer type}\label{Section classes of categories}
We will now solve the equations derived in the previous section. We will show that the possible categories can be distinguished by the values for $f$ and $f'$, the values for $e$ and $e'$, whether $\lambda'=\lambda$ or $\lambda'\not= \lambda$ and whether $\parS $ is zero or non-zero, and the parity of the cup and cap.
 
We will use the notation $\cC^{f,f'}_{\lambda',\parS }(\epsilon,e,e')$ for the category $\cC$ where we plug in the values for the occurring parameters. For example, $\cC^{0,b}_{\lambda,0}(-,1,1)$ has $f=0$, $f'=b$, $\lambda'=\lambda$, $\parS =0$, $e=e'=1$ and the cup and cap are odd morphisms.  We will often drop the $(\epsilon,e,e')$ part in this notation.  

\subsection{The case \texorpdfstring{$f=f'=0$.}{f=f'=0.}}
Note that from $f'(ec\parS '+f\lambda)= a(b-f)(b-f')/\lambda$ in Equation \eqref{Equation f f'b}, we have $ab^2/\lambda=0$. Since $a$ and $\lambda$ are non-zero, we conclude that $b=0$. The other equations in Equation \eqref{Equation f f'b} are then equivalent with $c\parS =0$.  
It can be readily verified that the equations in Equation \eqref{Equation DEF} are trivially satisfied, while Equation \eqref{Equation rest} reduces to 
\begin{equation}\label{equation case f=f=0}
\begin{aligned}
c\parq =0, \quad  \parq =\parS 'e\lambda= \parS e' \lambda', \quad \lambda \parq  = a\parS' e, \quad \lambda' \parq  = a\parS e', \quad \parS a\lambda e'=\parS 'a \lambda' e, \\ a\delta=ee'a\delta, \quad \parS (\lambda'-\lambda)(1+\epsilon e^2). 
\end{aligned}
\end{equation}
Since $c\parS =0$, it is clear that $c\parq =0$ follows from $\parq =\parS 'e\lambda$, while  $\lambda \parq  = a\parS' e$ and $\lambda' \parq  = a\parS e'$ imply $\parS a\lambda e'=\lambda \lambda' \parq =\parS 'a \lambda' e$.
\subsubsection{The subcase \texorpdfstring{$\parS $}{S} non-zero and \texorpdfstring{$\lambda'=\lambda$}{l'=l}.}
If $\parS $ is non-zero, then $c\parS =0$ imply $c=0$, and $ee'=1$ by Equation \eqref{Equation S-non-zero}.
If moreover $\lambda=\lambda'$, then Equation \eqref{equation case f=f=0} reduces to
$\parq =\epsilon  e \parS \lambda=e'\parS \lambda=(a/\lambda) \epsilon e\parS= (a/\lambda) e' \parS $.
Thus  $a^2=\lambda$, $e'=\epsilon e$ and therefore also $e^2=\epsilon$. 
Summarising, we have the following proposition. 
\begin{proposition}
If $f=f'=0$, $\lambda'=\lambda$ and  $\parS $ is non-zero, we obtain the category $\mathcal{C}^{0,0}_{\lambda, \parS }(\epsilon,e,\epsilon e)$, with independent variables 
$\{ \lambda, \parS , \delta \}$ and where $e$ is a square root of $\epsilon$, while for the other variables we have
\begin{align*}
\lambda'=\lambda, \quad  a=\lambda^2, \quad b=0, \quad c=0 \quad \parS '=\epsilon \parS , \quad \parq = \epsilon\parS e \lambda,  \quad e' = \epsilon e,
\\  d=d'=f=f'=0, \quad
D=E=D'=E'=0,  \quad F= \epsilon e \lambda^2, \quad F'=e \lambda^2.
\end{align*}
\end{proposition}
\subsubsection{The subcase \texorpdfstring{$\parS $}{S} non-zero and \texorpdfstring{$\lambda'\not=\lambda$}{l' not =l}.}
On the other hand, if $\parS $ is non-zero but $\lambda'\not=\lambda$, then from Equation \eqref{Equation mu not lambda}, we obtain $\lambda'=-\lambda$,  $a=\lambda^2$ and $\delta=0$.  
Equation \eqref{equation case f=f=0} will then be satisfied if $e'=-\epsilon e$ and  $e^2 =-\epsilon$. 
Hence, we obtain the following proposition. 
\begin{proposition}
If $f=f'=0$, $\lambda'\not=\lambda$ and  $\parS $ is non-zero, we obtain the category $\mathcal{C}^{0,0}_{-\lambda, \parS }(\epsilon,e,-\epsilon e)$, with independent variables 
$\{ \lambda, \parS \}$ and where $e$ is a square root of $-\epsilon$, while for the other variables we have
\begin{align*}
\lambda'=-\lambda, \quad \delta=0, \quad  a=\lambda^2, \quad b=0, \quad c=0 \quad \parS '=\epsilon \parS , \quad \parq = \epsilon \parS e \lambda, \quad e' = -\epsilon e,
 \\d=d'=f=f'=0,  \quad 
D=E=D'=E'=0, \quad F= -\epsilon e \lambda^2, \quad F'=e \lambda^2.
\end{align*}
\end{proposition}
\subsubsection{The subcase \texorpdfstring{$\parS =0$}{S=0} and $\texorpdfstring{\lambda'\not=\lambda$}{l' not= l}.}
Assume $\parS =0$, then it is clear that $\parq =0$ and Equation \eqref{equation case f=f=0}  reduces to $\delta (ee'-1)=0$. If $\lambda \not= \lambda'$, then Equation \eqref{Equation mu not lambda} implies $\delta=c=0$ and $a=\lambda^2$, $\lambda'=-\lambda$. The only restrictions  on $e$ or $e'$ are then $e^4=e'^4=1$ by Equation \eqref{Equation e,e'}.
We conclude the following.  
\begin{proposition}
If $f=f'=0$, $\lambda'\not=\lambda$ and  $\parS =0$, we obtain the category $\mathcal{C}^{0,0}_{-\lambda, 0}(\epsilon,e,e')$, with independent variable 
$\{ \lambda \}$ and $e^4=e'^4=1$ while for the other variables, we have
\begin{align*}
\lambda'=-\lambda, \quad \delta=0, \quad  a=\lambda^2, \quad b=0, \quad c=0 \quad \parS '= \parS =0, \quad \parq = 0,  \\d=d'=f=f'=0,  \quad 
D=E=D'=E'=0, \quad F= e^3 \lambda^2, \quad F'=e'^3 \lambda^2.
\end{align*}
\end{proposition}
\subsubsection{The subcase \texorpdfstring{$\parS =0$}{S=0} and \texorpdfstring{$\lambda'=\lambda$}{l=l'}.}
If $\parS =0$ and $\lambda'=\lambda$, then $a=\lambda^2-c\delta$ by Equation \eqref{Equation lambda=mu}. By Equation \eqref{Equation e,e'} we have $e^4=e'^4=1$. If furthermore $c$ or $\delta$ is non-zero, we have $ee'=1$. 
\begin{proposition}
If $f=f'=0$, $\lambda'=\lambda$ and  $\parS =0$, we obtain the category $\mathcal{C}^{0,0}_{\lambda, 0}(\epsilon,e,e')$, with independent variables 
$\{ \lambda,c,\delta \}$ and $e^4=e'^4=1$. Furthermore, if $c$ is non-zero or $\delta$ is non-zero then $e'=e^3$. For the other variables we have
\begin{align*}
\lambda'=\lambda, \quad  a=\lambda^2-c\delta, \quad b=0, \quad \parS '= \parS =0, \quad \parq = 0,  \\d=d'=f=f'=0,  \quad 
D=E=D'=E'=0, \quad F= e^3 \lambda^2, \quad F'=e'^3 \lambda^2.
\end{align*}
\end{proposition}

\subsection{The case \texorpdfstring{$f=f'=b\not=0$}{f=f'=b non-zero}}
If $f=f'=b\not=0$, then Equation \eqref{Equation f f'b} reduces to $e^2= e'^2=\epsilon $ and $ec\parS '= -b \lambda$ and $e'c\parS = -b\lambda'$. Hence, we immediately conclude that $\parS $ and $c$ are non-zero since $\lambda$ and $b$ are non-zero. Then Equation \eqref{Equation c non-zero} will be satisfied if $\lambda'=\lambda$ and $ee'=1$, or equivalently $e'=\epsilon e$. Equation \eqref{Equation S-non-zero} and Equation \eqref{Equation DEF} are then also trivially satisfied. Equation \eqref{Equation rest} will be satisfied if $\parq = \epsilon e \parS (\lambda-b)+b\delta$. Furthermore, since $\lambda'=\lambda$, we have $a=\lambda^2-b\lambda -c\delta$. 
We conclude the following. 
\begin{proposition}
Assume $f=f'=b\not=0$. Then we must have $\lambda'=\lambda$ and that $\parS $ and $c$ are non-zero. We obtain the category $\mathcal{C}^{b,b}_{\lambda, \parS }(\epsilon,e,\epsilon e)$, with independent variables 
$\{ \lambda,b,\parS ,\delta \}$ and $e^2=\epsilon$. For the other variables we have
\begin{align*}
\lambda'=\lambda, \quad  a=\lambda^2-b\lambda+e b\lambda\delta/\parS ,  \quad c= -e b\lambda/\parS \quad \parS '=\epsilon \parS , \quad \parq = \epsilon e \parS (\lambda-b)+b\delta,\\ 
 e'=\epsilon e, \quad d=-eb, \quad d'=-\epsilon e b, \quad f=f'=b,  \\
D=E=D'=E'=0, \quad F=\epsilon e a , \quad F'=e a .
\end{align*}
\end{proposition}
\subsection{The case \texorpdfstring{$f=0$, $f'=b\not=0$}{f=0, f'=b non-zero}.}
Equation \eqref{Equation f f'b} is satisfied if $c\parS =c\delta=0$ and  $e^2=e'^2=-\epsilon$.
From Equation \eqref{Equation c non-zero} we can conclude that $c=0$ since otherwise we would have $f=f'$.
Equation \eqref{Equation DEF} is trivially satisfied, while Equation \eqref{Equation rest}  reduces to
\begin{equation}\label{equation case f=0, f'=b}
\begin{aligned}
\parq =\epsilon e\parS (\lambda-b) = e'\parS \lambda' + b\delta,  \quad (\lambda'-b)\parq = \frac{a}{\lambda} b\delta +a\parS e' \\
b\parq \lambda+ab\delta +\parS a\lambda e'= \parS 'a\lambda' e, \quad \delta (ee'-1)=0, \quad  \parS (1-ee')=0,
\end{aligned}
\end{equation}
where we frequently used $\lambda^2-b\lambda=\lambda'^2-b\lambda' = a$. 
Note that using $\parq = e'\parS \lambda' + b\delta$, we see that $(\lambda'-b)\parq = \frac{a}{\lambda} b\delta +a\parS e'$ is equivalent to $a\delta= \lambda(\lambda'-b)\delta$. This is always satisfied since if $\lambda=\lambda'$, we have $ \lambda(\lambda'-b)=\lambda^2-b\lambda=a$ and if $\lambda\not= \lambda'$, then $\delta=0$.
\subsubsection{The subcase \texorpdfstring{$\parS $}{S} is zero.}
Assume $\parS $ is zero, then the equations in \eqref{equation case f=0, f'=b} reduces to $\parq =0$ and $\delta=0$. 
We conclude the following
\begin{proposition}
Let $f=0$, $f'=b\not=0$ and $\parS =0$. 
If $\lambda' =\lambda$, we obtain the category $\mathcal{C}^{0,b}_{\lambda, 0}(\epsilon,e,e')$, while for $\lambda'\not=\lambda$, we have $\lambda'=b-\lambda$ and we obtain the category $\mathcal{C}^{0,b}_{b-\lambda, 0}(\epsilon,e,e')$. In both cases, we have independent variables $\{ \lambda,b\}$ and $e^2=e'^2=-\epsilon$. For the other variables, we have
\begin{align*}
 a=\lambda^2-b\lambda,  \quad c=0 \quad \parS '=\parS =0, \quad \parq =0, \quad \delta=0\\ 
 d=-eb, \quad d'=0, \quad f=0, \quad f'=b,  \\
D=(\lambda-b)b, \quad  E=b, \quad D'=E'=0, \quad F=-\epsilon e a , \quad F'=-\epsilon e' a .
\end{align*}
\end{proposition}
\subsubsection{The subcase \texorpdfstring{$\parS $}{S} is non-zero.}
Assume $\parS $ is non-zero, then $e'=1/e= -\epsilon e$. If we then set $\parq = \epsilon e\parS (\lambda-b)$ and $\delta=\epsilon e \parS (\lambda+\lambda'-b)/b$, we see that Equation \eqref{equation case f=0, f'=b} is satisfied. Note that $\delta=0$ for $\lambda' = b-\lambda$.  So we have proven the following proposition. 
\begin{proposition}
Let $f=0$, $f'=b\not=0$ and $\parS $ non-zero. If $\lambda' =\lambda$, we obtain the category $\mathcal{C}^{0,b}_{\lambda, \parS }(\epsilon,e,-\epsilon e)$, while for $\lambda'\not=\lambda$, we have $\lambda'=b-\lambda$ and we obtain the category $\mathcal{C}^{0,b}_{b-\lambda, 0}(\epsilon,e,-\epsilon e)$. In both cases, we have  independent variables 
$\{ \lambda,b, \parS \}$ and $e^2=-\epsilon$.  For the other variables, we have
\begin{align*}
 a=\lambda^2-b\lambda,  \quad c=0 \quad \parS '=\epsilon \parS , \quad \parq =\epsilon e \parS (\lambda-b),  \quad \delta=\epsilon e \parS (\lambda+\lambda'-b)/b\\ 
e'=-\epsilon e, \quad  d=-eb, \quad d'=0, \quad f=0, \quad f'=b,  \\
D=(\lambda-b)b, \quad  E=b, \quad D'=E'=0, \quad F=-\epsilon e a , \quad F'=ea .
\end{align*}
\end{proposition}

\subsection{The case \texorpdfstring{$f=b\not=0$, $f'=0$}{f=b non-zero, f'=0}}
This case is similar to the case $f=0$ and $f'=b$ with the accents switched.
So we obtain the following results. 
\begin{proposition}
Let $f'=0$, $f=b\not=0$ and $\parS =0$. 
If $\lambda' =\lambda$, we obtain the category $\mathcal{C}^{b,0}_{\lambda, 0}(\epsilon,e,e')$, while for $\lambda'\not=\lambda$, we have $\lambda'=b-\lambda$ and we obtain the category $\mathcal{C}^{b,0}_{b-\lambda, 0}(\epsilon, e,e')$. In both cases, we have  independent variables 
$\{ \lambda,b\}$ and $e^2=e'^2=-\epsilon$. For the other variables, we have
\begin{align*}
 a=\lambda^2-b\lambda,  \quad c=0 \quad \parS '=\parS =0, \quad \parq =0, \quad \delta=0\\ 
 d=0, \quad d'=-e'b, \quad f=b, \quad f'=0,  \\
D=E=0, \quad D'=(\lambda'-b) b, \quad  E'=b, \quad F=-\epsilon e a , \quad F'=-\epsilon e' a .
\end{align*}

Let $f'=0$, $f=b\not=0$ and $\parS $ non-zero.  If $\lambda' =\lambda$, we obtain the category $\mathcal{C}^{b,0}_{\lambda, \parS }(\epsilon,e,-\epsilon e)$, while for $\lambda'\not=\lambda$, we have $\lambda'=b-\lambda$ and we obtain the category $\mathcal{C}^{b,0}_{b-\lambda, 0}(\epsilon,e,-\epsilon e)$. We have  independent variables 
$\{ \lambda,b, \parS \}$ and $e^2=-\epsilon$. For the other variables we have
\begin{align*}
 a=\lambda^2-b\lambda,  \quad c=0 \quad \parS '=\epsilon \parS , \quad \parq =-\epsilon e \parS (\lambda'-b),  \quad \delta=-\epsilon e \parS (\lambda+\lambda'-b)/b\\ 
e'=-\epsilon e, \quad  d'=\epsilon e b, \quad d=0, \quad f'=0, \quad f=b,  \\
D=E=0, \quad D'=(\lambda'-b)b, \quad  E'=b, \quad  F=-\epsilon e a , \quad F'=e a .
\end{align*}
\end{proposition}
\subsection{Summary table}

\renewcommand{\arraystretch}{1.3}

We have summarized the results of this section in Table \ref{summary table}. 

Note that if we work over a field $\mK$ where $-1$ is not a square, then there are no odd versions of the categories $\cC^{0,0}_{\lambda,\parS }$ and $\cC^{b,b}_{\lambda,\parS }$, and there are no even versions of the categories $\cC^{b,0}_{\lambda,\parS }$, 
$\cC^{b,0}_{b-\lambda,\parS }$, $\cC^{b,0}_{\lambda,0}$, $\cC^{b,0}_{b-\lambda,0}$, 
$\cC^{0,b}_{\lambda,\parS }$, $\cC^{0,b}_{b-\lambda,\parS }$, $\cC^{0,b}_{\lambda,0}$, $\cC^{0,b}_{b-\lambda,0}$ and $\cC^{0,0}_{-\lambda,\parS }$.

We also remark that taking the limit  $b$ going to $0$ in $\cC^{b,b}_{\lambda,\parS}(\epsilon, e,\epsilon e)$ leads to the category $\cC^{0,0}_{\lambda,\parS}(\epsilon, e,\epsilon e)$. However, we can not take the limit for $\parS$ going to zero in $\cC^{b,b}_{\lambda,\parS}(\epsilon, e,\epsilon e)$ since then $a$ and $c$ would become infinity. 
We also have the following limits
\begin{align*}
\lim_{\parS\to 0}  \cC^{0,0}_{\lambda,\parS}(\epsilon,e,\epsilon e) &= \cC^{0,0}_{\lambda,0}(\epsilon, e,\epsilon e) \\
\lim_{\parS\to 0}  \cC^{b,0}_{\lambda,\parS}(\epsilon,e,-\epsilon e) &= \cC^{b,0}_{\lambda,0}(\epsilon, e,-\epsilon e) \\
\lim_{b\to 0}  \cC^{b,0}_{\lambda,0}(\epsilon,e,e') &= \cC^{0,0}_{\lambda,0}(\epsilon, e,e') \\
\lim_{\parS\to 0}  \cC^{b,0}_{b-\lambda,\parS}(\epsilon,e,-\epsilon e) &= \cC^{b,0}_{b-\lambda,0}(\epsilon, e,-\epsilon e) \\
\lim_{b\to 0}  \cC^{b,0}_{b-\lambda,0}(\epsilon,e,e') &= \cC^{0,0}_{-\lambda,0}(\epsilon, e, e') \\
\lim_{b\to 0}  \cC^{b,0}_{b-\lambda,\parS}(\epsilon,e,-\epsilon e) &= \cC^{0,0}_{-\lambda,\parS}(\epsilon, e,-\epsilon e) \\
\lim_{\parS\to 0}  \cC^{0,0}_{-\lambda,\parS}(\epsilon,e,-\epsilon e) &= \cC^{0,0}_{-\lambda,0}(\epsilon, e,-\epsilon e),
\end{align*}
and similar limits for $\cC^{0,b}_{\lambda',\parS}$. 
Note that in the limit $\lim_{\parS\to 0}  \cC^{0,0}_{\lambda,\parS}(\epsilon,e,\epsilon e)$ we get the category $\cC^{0,0}_{\lambda,0}(\epsilon, e,\epsilon e)$ where the independent variable $c$ is zero. Similar for  $\lim_{b\to 0}  \cC^{b,0}_{\lambda,0}(\epsilon,e,e')$ where in $\cC^{0,0}_{\lambda,0}(\epsilon, e,e')$ the independent variables $c$ and $\delta$ are zero.
Remark also that the limit $\lim_{b\to 0} \cC^{b,0}_{\lambda,\parS}(\epsilon,e,-\epsilon e )$ does not exist since $\delta$ goes to infinity.

\begin{landscape}
\begin{table}
\begin{tabular}{|c|c|c|c|c|c|c|c|c|c|c|c|c|c|c|c|c|}
\hline 
Name &  $\lambda$  & $\lambda'$ & $\parS $ & $b$ & $c$ & $\delta$ & $\parq $   &  d & e &  d' & e'  & D & E & D' & E' \\ 
\hline 
$\cC^{b,0}_{\lambda,\parS }$ &  \color{red} $\lambda$ & $\lambda$  & \color{red} $\parS $ 
& \color{red} $b$ & $0$ &$-e\epsilon \parS \frac{2\lambda-b}{b}$ & $-\epsilon e \parS (\lambda-b)$   
& $0$   & $e^2=-\epsilon$ & $\epsilon e  b$ & $-\epsilon e$ & $0$ & $0$ & $b(\lambda-b)$ & $b$
\\
\hline 
$\cC^{b,0}_{b-\lambda,\parS }$ &  \color{red} $\lambda$ & $b-\lambda$ & \color{red} $\parS $ &  \color{red} $b$ & $0$ &$0$ & $\epsilon e \parS \lambda$  
 & $0$   & $e^2=-\epsilon$ & $\epsilon e  b$ & $-\epsilon e$ & $0$ & $0$ & $-b\lambda$ & $b$
 \\
\hline 
$\cC^{b,0}_{\lambda,0}$ & \color{red} $\lambda$ & $\lambda$  &$0$  &  \color{red} $b$ & $0$  & $0$ & $0$
& $0$   & $e^2=-\epsilon$ & $- e' b$ & $e'^2=-\epsilon$ & $0$ & $0$ & $b(\lambda-b)$ & $b$ 
\\
\hline 
$\cC^{b,0}_{b-\lambda,0}$ & \color{red} $\lambda$ & $b-\lambda$ &$0$ & \color{red} $b$ & $0$ & $0$ & $0$ 
& $0$   & $e^2=-\epsilon$ & $- e' b$ & $e'^2=-\epsilon$ & $0$ & $0$ & $-b\lambda$ & $b$ 
 \\
\hline \hline 
$\cC^{0,b}_{\lambda,\parS }$ & \color{red} $\lambda$ & $\lambda$ & \color{red} $\parS $  & \color{red} $b$ & $0$  & $e\epsilon \parS \frac{2\lambda-b}{b}$  & $\epsilon e \parS (\lambda-b)$ 
&  $-eb$   & $e^2=-\epsilon$ & $0$ & $-\epsilon e$ & $b(\lambda-b)$ & $b$ & $0$ & $0$
 \\
\hline 
$\cC^{0,b}_{b-\lambda,\parS }$ & \color{red} $\lambda$  & $b-\lambda$ & \color{red} $\parS $ & \color{red} $b$ & $0$   &$0$ & $\epsilon e \parS (\lambda-b)$  
&  $-eb$   & $e^2=-\epsilon$ & $0$ & $-\epsilon e$ & $b(\lambda-b)$ & $b$ & $0$ & $0$
\\
\hline 
$\cC^{0,b}_{\lambda,0}$ &  \color{red} $\lambda$ & $\lambda$ &$0$ &  \color{red} $b$ & $0$ & $0$ & $0$ 
 &  $-eb$   & $e^2=-\epsilon$ & $0$ & $e'^2=-\epsilon$ & $b(\lambda-b)$ & $b$ & $0$ & $0$
  \\
\hline 
$\cC^{0,b}_{b-\lambda,0}$ &  \color{red} $\lambda$ & $b-\lambda$ &$0$ & \color{red} $b$ & $0$ & $0$ & $0$ 
&  $-eb$   & $e^2=-\epsilon$ & $0$ & $e'^2=-\epsilon$ & $b(\lambda-b)$ & $b$ & $0$ & $0$ 
 \\
\hline
\hline 
$\cC^{b,b}_{\lambda,\parS }$ & \color{red} $\lambda$ & $\lambda$  & \color{red} $\parS $ & \color{red} $b$ & $-e\lambda \frac{b}{\parS} $ &\color{red} $\delta$ & $\epsilon \parS e (\lambda-b)+b\delta$ 
& $-eb$   & $e^2=\epsilon$ & $-\epsilon e b$ & $\epsilon e$ & $0$ & $0$ & $0$ & $0$
  \\
\hline 
\hline
$\cC^{0,0}_{\lambda,\parS }$ &  \color{red} $\lambda$ & $\lambda$ & \color{red} $\parS $ & $0$ & $0$  & \color{red} $\delta$  & $ \epsilon \parS  e \lambda$
&  $0$   & $e^2=\epsilon$ & $0$ & $\epsilon e$ & $0$ & $0$ & $0$ & $0$
 \\ 
\hline 
$\cC^{0,0}_{-\lambda,\parS }$ & \color{red} $\lambda$   & $-\lambda$ & \color{red} $\parS $ &  $0$ & $0$  & $0$ & $\epsilon \parS e \lambda$ 
& $0$   & $e^2=-\epsilon$ & $0$ & $-\epsilon e$ & $0$ & $0$ & $0$ & $0$  
 \\ 
\hline 
$\cC^{0,0}_{\lambda,0}$ &  \color{red} $\lambda$  & $\lambda$ & $0$  & $0$ & \color{red} $c$ & \color{red} $\delta$ & $ 0$  
&  $0$   & $e^4=1$ & $0$ & $e'^4=1$ & $0$ & $0$ & $0$ & $0$
 \\ 
\hline 
$\cC^{0,0}_{-\lambda,0}$ & \color{red} $\lambda$ & $-\lambda$ & $0$ & $0$ & $ 0$ &  $0$ & $0$ 
&  $0$   & $e^4=1$ & $0$ & $e'^4=1$ & $0$ & $0$ & $0$ & $0$
\\ 
\hline 
\end{tabular} \caption{The values of the parameters for the possible categories of Brauer type $\cC^{f,f'}_{\lambda',\parS}(\epsilon,e,e')$. The independent parameters are shown in red. Note that the independent parameters $\lambda$, $b$ and $\parS $ are always non-zero, while the independent parameters $\delta$ and $c$ are allowed to be zero. If $-1$ is not a square, all categories but $\cC^{0,0}_{-\lambda,0}$ and $\cC^{0,0}_{\lambda,0}$ only exist for one type of parity.  \label{summary table}}
\end{table}
\end{landscape}

\section{Bases for the categories}\label{Section basis}
We claim that the Brauer diagrams represented by their standard expression give  a basis for the hom-spaces of the category $\mathcal{C}$. We have already shown that these diagrams are a spanning set in Theorem \ref{theorem simplifying}. 

To prove the linear independence of our proposed basis we will use a trick described in \cite[Section 0.3]{FDAandQuantumGroups} adapted to a categorical setting. It works as follows. Assume we want to show that a spanning set $(x_i)_i$ in a unital algebra $A$ is a linearly independent set, where we furthermore assume $x_0= \id$. We construct a free module $V$ by formally taking linear combinations of the set $(X_i)_i$, where each $X_i$ corresponds to $x_i$. By definition $(X_i)_i$ is a linear independent set. Then we define an action $F\colon A \to \End(V)$ by setting $F(a) X_i = \sum_j \lambda_j X_j$ if $ax_i = \sum_{j} \lambda_j x_j$. The difficult part is showing that $F$ is well-defined. But if $F$ is indeed well-defined, then linear independence for $(x_i)_i$ follows immediately since $\sum_i \mu_i x_i= 0$ implies $F(\sum_i \mu_i x_i) X_0 = \sum_i \mu_i X_i =0.$ We conclude that all $\mu_i$ are zero since the $X_i$ are linearly independent.  
\begin{theorem} \label{Theorem basis}
Let $\cC$ be a category of Brauer type as in definition \ref{Definition category of Brauer type}. Then the $(m,n)$-Brauer diagrams depicted using their standard expression form a basis of $Hom_{\cC}(m,n)$. 
\end{theorem}
\begin{proof}
We have already shown that they form a spanning set in Theorem \ref{theorem simplifying}. Let us now prove linear independence using an adapted version of the trick described above. Let $(X_i)_i$ be the set of all Brauer diagrams and $V$ the free module obtained by taking linear combinations of these Brauer diagrams. Let $A$ be the algebra consisting of the morphisms of the category $\mathcal{C}$. This means that if morphisms are compatible, then their product in $A$ is given by the composition of morphisms in the category, while if two morphisms are not compatible their product is by definition zero. 

We define an action $F$ of $A$ on $V$ as follows. Each Brauer diagram $X$ corresponds by Proposition \ref{Proposition standard expression} to a unique standard expression $x$ which is a morphism in $A$. Moreover, from Theorem \ref{theorem simplifying}, we know that for each fundamental diagram $a$, the product $ax$ can be simplified using the defining relations to a linear combination $\sum \lambda_i x_i$ of standard expressions of Brauer diagrams. From Section \ref{Section rewritable diagrams}, we know that this simplification is unique. We can thus define $F(a)X$ by $\sum_i \lambda_i X_i$. Since the fundamental diagrams generate $A$ and we used the defining relations in the definition of $F$, this gives a well-defined action of the whole $A$ on $V$. We do not have an identity morphism in $A$, but we can use the identity $(m,m)$-Brauer diagram $\id_m$ instead. This is the diagram which consists of $m$ non-crossing propagating lines. So assume $a=\sum \mu_j x_j =0$ where all $x_j$ are standard expressions in $\Hom_{\cC}(m,n)$. Then we see that 
$0=F(a) X_{\id_m}= \sum \mu_j X_j$ since $x_j \id_m=x_j$. Since the $X_i$ are linearly independent, all $\mu_j$ are zero, which concludes the proof. 
\end{proof}

\section{Scaling and flipping} \label{Section scaling}
\subsection{Rescaling}
In this section, we will show that, up to a monoidal isomorphism, it is always possible to rescale   either $\lambda$, $b$ or $a$ to $1$ and to rescale either the parameter $\parS $, $\delta$, $c$ or $\parq $ to $1$ if they are non-zero.

Let $\cC$ and $ \widetilde{\cC}$ be two categories of Brauer type. 
We define a (strict) monoidal superfunctor $F_{\alpha,\beta,\gamma}\colon \cC \to \widetilde{\cC}$, where $\alpha, \beta$ and $\gamma$ are non-zero scalars. 
On the objects $F_{\alpha,\beta,\gamma}$ act as the identity, and on the generating morphism $F_{\alpha,\beta,\gamma}$ acts by scalar multiplication with  $\alpha,\beta$ and $\gamma$:
\begin{align*}
F_{\alpha,\beta,\gamma}\left(\di{0.5\sca}{\ca}\right) = \alpha \,\di{0.5\sca}{\ca}\, , \quad
F_{\alpha,\beta,\gamma}\left(\di{0.5\sca}{\cu}\right) = \beta\, \di{0.5\sca}{\cu}\, , \quad 
F_{\alpha,\beta,\gamma}\left(\di{0.5\sca}{\ocr}\right) = \gamma \,\di{0.5\sca}{\ocr}.
\end{align*}
\begin{lemma}
The functor $F_{\alpha,\beta,\gamma}$ is a well-defined isomorphism if the parameters of $\cC$ and $\widetilde{\cC}$ are related by
\begin{align*}
\gamma \tilde{\lambda}= \lambda, \quad \gamma \tilde{\lambda'} = \lambda', \quad \alpha\beta \tilde{\parS } = \parS , \quad \alpha\beta \tilde{\parS '} = \parS ', \quad \alpha\beta \tilde{\delta} = \delta, \quad \alpha\beta\gamma \tilde{\parq } = \parq ,  \\
\gamma^2 \tilde{a} = a, \quad \gamma \tilde{b} = b, \quad \gamma^2 \tilde{c} = \alpha\beta c, \\
\gamma \tilde{d} = d, \quad \gamma \tilde{d'} = d', \quad \tilde{e}=e, \quad \tilde{e'}=e', \quad \gamma \tilde{f} = f, \quad \gamma \tilde{f'} = f', \quad  \\
\gamma^2 \tilde{D}= D, \quad \gamma^2 \tilde{D'}= D', \quad \gamma \tilde{E}= E, \quad \gamma \tilde{E'}= E', \quad \gamma^2 \tilde{F}= F \quad \gamma^2 \tilde{F'}= F'.
\end{align*}
\end{lemma}

\begin{proof}
Since we defined $F_{\alpha,\beta,\gamma}$ on the generating morphisms, we only have to check that $F_{\alpha,\beta,\gamma}$ respects the defining relations of the categories. Applying $F_{\alpha,\beta,\gamma}$ on the defining relations from Section \ref{Section relations} gives us exactly the relations in the lemma. Note that the relations in this lemma are consistent with the relations between the parameters of Section \ref{Section rewritable diagrams}. For instance 
$a= \lambda^2-b\lambda -c\delta$. Applying the relations from the lemma, this results in 
$\gamma^2 \tilde{a} = (\gamma \tilde{\lambda})^2 - \gamma \tilde{b} \gamma \tilde{\lambda} - \frac{\gamma^2}{\alpha\beta} \tilde{c} \alpha \beta \tilde{\delta}$. 
It is clear that the inverse is given by $F_{1/\alpha,1/\beta,1/\gamma}$.
\end{proof}

Since $F_{\alpha,\beta,\gamma}$ is an isomorphism with inverse $F_{1/\alpha,1/\beta,1/\gamma}$, we see that we can use it to rescale two parameters in our categories of Brauer type. Remark that we can only rescale two parameters and not three since $\alpha$ and $\beta$ always occur together as the product $\alpha\beta$. To obtain the connection with categories occurring in the existing literature in Section \ref{Section connection with existing categories} we will have to rescale in such a way that $\parS =1$ and $a=1$. 

\subsection{Vertical flipping}
We also have a contravariant monoidal functor by exchanging the roles of the cup and the cap. This corresponds to vertically flipping a diagram. 
 \begin{proposition}
Let $\cC$ and $\widetilde{\cC}$ be two categories of Brauer type whose parameters are related as follows:
\begin{align*}
 \tilde{\lambda}= \lambda', \quad  \tilde{\lambda'} = \lambda, \quad \tilde{\parS } = \parS ', \quad  \tilde{\parS '} = \parS , \\ \tilde{\delta} = \delta, \quad  \tilde{\parq } = \parq ,  \quad 
 \tilde{a} = a, \quad  \tilde{b} = b, \quad  \tilde{c} = c, \\
\tilde{d} = d', \quad \tilde{d'} = d, \quad \tilde{e}=e', \quad \tilde{e'}=e, \quad  \tilde{f} = f', \quad  \tilde{f'} = f, \quad  \\
 \tilde{D}= D', \quad  \tilde{D'}= D, \quad \tilde{E}= E', \quad  \tilde{E'}= E, \quad  \tilde{F}= F' \quad  \tilde{F'}= F.
\end{align*}
Then the monoidal contravariant superfunctor $F_{v-flip}$ defined by $F_{v-flip}(\di{0.5\sca}{\ocr})=\di{0.5\sca}{\ocr}$, $F_{v-flip}(\di{0.5\sca}{\ca})=\di{0.5\sca}{\cu}$ and $F_{v-flip}(\di{0.5\sca}{\cu})=\di{0.5\sca}{\ca}$ is a well-defined isomorphism between $\cC$ and $\widetilde{\cC}$. 
\end{proposition}
\begin{proof}
We defined $F_{v-flip}$ on the generators and a straightforward verification shows that $F_{v-flip}$ respects the defining relations for the given parameters. Applying  $F_{v-flip}$ twice is clearly the identity functor, hence $F_{v-flip}$ is an isomorphism. 
\end{proof}

\subsection{Horizontal flipping}
We can also consider the functor which flips diagrams horizontally. This will, however, not be a monoidal functor. 
 \begin{proposition}
Let $\cC$ and $\widetilde{\cC}$ be two categories of Brauer type whose parameters are related as follows:
\begin{align*}
 \tilde{\lambda}= \lambda, \quad  \tilde{\lambda'} = \lambda', \quad \tilde{\parS } = \parS ', \quad  \tilde{\parS '} = \parS , \\ \tilde{\delta} = \delta, \quad  \tilde{\parq } = \parq +e\parS'(\lambda'-\lambda)  ,  \quad 
 \tilde{a} = a, \quad  \tilde{b} = b, \quad  \tilde{c} = c, \\
\tilde{d} =  \frac{d'}{ee'}, \quad \tilde{d'} = \frac{d}{ee'}, \quad \tilde{e}=1/e, \quad \tilde{e'}=1/e', \quad  \tilde{f} = f', \quad  \tilde{f'} = f, \quad  \\
 \tilde{D}= D' \lambda'/\lambda , \quad  \tilde{D'}= D\lambda/\lambda' , \quad \tilde{E}= E', \quad  \tilde{E'}= E, \quad  \tilde{F}= ee' F' \quad  \tilde{F'}= ee' F.
\end{align*}
Then the superfunctor $F_{h-flip}$ defined as the identity on the objects and the generating morphisms and satisfying by $F_{h-flip}(X\otimes Y) = (-1)^{\abs{X}\abs{Y}} F_{h-flip}(Y) \otimes F_{h-flip}(X)$ is a well-defined isomorphism between $\cC$ and $\widetilde{\cC}$. We then also have $\cC^{\op}\cong \widetilde{\cC}$.
\end{proposition}
\begin{proof}
The functor $F_{h-flip}$ is the identity on objects and the generating morphisms, so we only have to verify that it respects the relations. 
For example, applying $F_{h-flip}$ on the delooping relation gives us
\begin{align*}
\parq  \; \di{0.5\sca}{\li}  =F_{h-flip}\left( \di{0.5\sca}{
\cu \ocr[\br][0] \li \ca[0][\h] \dli[2*\br][0] \uli[2*\br][\h] 
} \right) =  \di{-0.5\sca,0.5\sca}{
\cu \ucr[\br][0] \li \ca[0][\h] \dli[2*\br][0] \uli[2*\br][\h] 
}= \left(\tilde{d}\tilde{\parS'} + \tilde{e}\tilde{\lambda'}\tilde{\parS'} +\tilde{f}\tilde{\delta} \right)\; \di{0.5\sca}{\li}, 
\end{align*}
where we used the sliding relation in $\widetilde{\cC}$. From the relations between the parameters of the categories $\cC$ and $\widetilde{\cC}$, we see that $\rho=\tilde{d}\tilde{\parS'} + \tilde{e}\tilde{\lambda'}\tilde{\parS'} +\tilde{f}\tilde{\delta}$ is equivalent to $\parq   =  \parS (d'+e'\lambda') + f'\delta$ . This last expression holds by Equation  \eqref{Equation rest}. The other relations can be similarly verified.

 Applying  $F_{h-flip}$ twice is clearly the identity functor, hence $F_{h-flip}$ is an isomorphism. Combining $F_{h-flip}$ with the construction of the monoidal opposite is clearly the identity, so we can conclude that $\cC^{\op}\cong \widetilde{\cC}$.
\end{proof}

From this, we can immediately conclude that the Brauer category, BWM-category and the periplectic Brauer category are their own monoidal opposite. This does not hold for the periplectic $q$-Brauer category.

\begin{corollary}\label{Corollary Monoidal opposites}
The categories $\cC^{b,b}_{\lambda,\parS} (+,1,1)$, $\cC^{0,0}_{\lambda,\parS} (+,1,1)$ and $\cC^{0,0}_{\lambda,\parS} (-,1,1)$ are each their own monoidal opposite, while the categories  $\cC^{b,0}_{b-\lambda,\parS} (-,1,1)$ and $\cC^{0,b}_{b-\lambda,\parS} (-,1,1)$ are each other monoidal opposites. 
\end{corollary}
\begin{proof}
For $\cC^{b,b}_{\lambda,\parS} (+,1,1)$ and $\cC^{0,0}_{\lambda,\parS} (+,1,1)$ this follows immediately from the previous theorem. For $\cC^{b,0}_{b-\lambda,\parS} (-,1,1)$ we have to combine $F_{h-flip}$ with the rescaling $F_{1,1,\epsilon}$ to obtain $\cC^{0,b}_{b-\lambda,\parS} (-,1,1)$. Similarly, combining $F_{h-flip}$ with the rescaling $F_{1,1,\epsilon}$ gives $\cC^{0,0}_{\lambda,\parS} (-,1,1) \cong (\cC^{0,0}_{\lambda,\parS} (-,1,1))^{\op}.$
\end{proof}

\section{Connection with existing categories}\label{Section connection with existing categories}
In this section, we will show the connection between the categories we introduced in this paper and known categories in the literature. 
\subsection{The Birman-Wenzl-Murakami category} \label{Subsection BWM category}
The Birman-Wenzl-Murakami algebra is a deformation of the Brauer algebra introduced by Birman and Wenzl in \cite{BirmanWenzl} and Murakami in \cite{Murakami}. 
We will use the definition of the BWM-algebra as stated in \cite{Morton}, which can also be found in \cite{LehrerZhang}
\begin{definition}[{\cite[Section 2.1]{Morton}, \cite[Section 8]{LehrerZhang}}]
The Birman-Wenzl-Murakami algebra $BWM_n$ is the unital, associative $\mK[v,v^{-1}, z, \delta]/ (v^{-1} - v - z(\delta - 1))$-algebra generated by elements $g_i^{\pm}$ and $e_i$ for $1 \leq i \leq n-1$ satisfying the following relations:
\begin{align*}
(g_i-g_i^{-1})&=z(1-e_i), \\ e_i^2 &= \delta e_i, \\e_i g_i &= v e_i = g_i e_i,   \\
g_i g_j &= g_j g_ i, \qquad \text{ for } |i-j|\geq 2  \\
g_i g_{i+1} g_i &= g_{i+1} g_i g_{i+1}, 
\\  e_{i+1} e_i e_{i+1}  = e_{i+1}, & \qquad e_i e_{i+1} e_i = e_i,  \\
g_i g_{i+1} e_i = e_{i+1} e_i, & \qquad  g_{i+1} g_i e_{i+1} = e_i e_{i+1} \\
e_i g_{i+1} e_i = v^{-1} e_i, & \qquad e_{i+1} g_i e_{i+1} = v^{-1} e_{i+1}. 
\end{align*}
\end{definition}
Note that the Brauer algebra is obtained from the BWM-algebra by setting $z=0$ and $v=1$. 

Although we could not find an explicit definition for the BWM-category in the literature, the definition of the Brauer category defined in \cite{LehrerZhang} can be easily deformed to obtain a BWM-category which has as endomorphism spaces the BWM algebras.

\begin{definition}[BWM-category] \label{Definition BWM}
The BWM-category $\mathcal{B}$ is a strict monoidal category generated by a single object $\bullet$ and the even morphisms
$\di{0.5\sca}{\ocr} \in \mathcal{B}(2,2)$, $\di{0.5\sca}{\ucr} \in \mathcal{B}(2,2)$,
$
\di{0.5\sca}{\ca} \in \mathcal{B}(0,2)$ and $\di{0.5\sca}{\cu} \in \mathcal{B}(2,0)$
subject to the following defining relations:
\begin{enumerate}
\item The Kauffman skein relation:  $\di{0.5\sca}{\ocr}-\di{0.5\sca}{\ucr} = z\; \di{0.5\sca}{\lili}-z\; \di{0.5\sca}{\cc}$,
\item The loop removing relation: $\di{0.5\sca}{\cu \ca}=\delta$,
\item The untwisting relations: $\di{0.5\sca}{\cu \ocr} = v\; \di{0.5\sca}{\cu}$ and $\di{0.5\sca}{\ca[0][\h] \ocr} = v\; \di{0.5\sca}{\ca}$,
\item The braid relations: $\di{0.5\sca}{\ocr \ucr[0][\h]}=\di{0.5\sca}{\lili} = \di{0.5\sca}{\ucr \ocr[0][\h]}$ and $\di{0.25\sca}{\ocr \li[2*\br][0] \li[0][\h] \ocr[\br][\h] \ocr[0][2*\h] \li[2*\br][2*\h]} = \di{0.25\sca}{\ocr[\br][0] \li[0][0] \li[2*\br][\h] \ocr[0][\h] \ocr[\br][2*\h] \li[0][2*\h]}$,
\item The snake relations: $\di{0.5\sca}{\ca \dli \cu[\br][0] \uli[2*\br][0]}=\di{0.5\sca}{\li}$ and $\di{0.5\sca}{\uli \cu \ca[\br][0] \dli[2*\br][0]}=\; \di{0.5\sca}{\li}$,
\item The tangle relation:  $\di{0.5\sca}{\dli[2*\br][0] \cu \ocr[\br][0] \li \li[2*\br][\h] \ocr[0][\h] }=  \di{0.5\sca}{\licu }$ and  $\di{0.5\sca}{\li \ocr[\br][0] \ocr[0][-\h] \li[2*\br][-\h] \dli[0][-\h] \cu[\br][-\h]}=\di{0.5\sca}{\cu \li[2*\br][-\h]}$,
\item Delooping relations: $\di{0.5\sca}{
\cu \ocr[\br][0] \li \ca[0][\h] \dli[2*\br][0] \uli[2*\br][\h] 
}= v^{-1} \; \di{0.5\sca}{\li}$ and $\di{0.5\sca}{ \dli \cu[\br][0] \ocr \ca[\br][\h] \uli[0][\h] \li[2*\br][0] } = v^{-1} \; \di{0.5\sca}{\li}$.
\end{enumerate}
\end{definition}

This BWM-category is equivalent to a category of Brauer type we constructed in this paper.
\begin{proposition}
Consider the category $\mathcal{C}=\mathcal{C}_{v,1}^{z,z}(+,1,1)$ where we scaled $\parS $ and $a$ to be one.  
This category is isomorphic to the BWM-category $\mathcal{B}$ defined in Definition \ref{Definition BWM}.
\end{proposition}
\begin{proof}
The category $\cC=\mathcal{C}_{v,1}^{z,z}$ has four independent parameters $v$, $\delta$, $\parS $ and $z$. We rescale $\parS $ to one, eliminating the parameter $\parS $. Further rescaling $a$ to one gives the relation $1=v^2-zv+zv\delta$ or equivalently $v^{-1}= v-z+z\delta$. This is the same relation we also have in $\mathcal{B}$ between $v,\delta$ and $z$.  Then the parameters of $\cC$ become
\begin{align*}
\lambda&=\lambda'=v, \quad \delta=\delta, \quad \parS =\parS '=1, \quad a=1, \quad b=z, \quad c=-zv, 
\\
\parq  &=v^{-1}, \quad d=d'=-z,\quad e=e'=1, \quad f=f'=z,  
\\
D&=E=D'= E'=0, \quad F=F'=1.
\end{align*}
 If we then compare the relation in $\cC$ with the relations in $\mathcal{B}$ we see that the relations which they do not share are the twisting, the sliding, upside-down sliding and upside-down pulling in $\cC$ and the Kauffman skein relation, the right tangle and right delooping relation in $\mathcal{B}$. 
Remark that applying the under-cross to the twisting relation in $\cC$ and then using untwisting, shows that twisting is equivalent to the Kauffman skein relation.
If we multiply sliding with $\di{0.5\sca}{\li \ocr[\br][0]}$ we obtain, using twisting and untwisting,
\begin{align*}
\di{0.5\sca}{\li \ocr[\br][0]} \cdot \di{0.5\sca}{ \ocr \dli \cu[\br][0] \li[2*\br][0] } 
&=-z \; \di{0.5\sca}{\cu \dli[2*\br][0] \li[0][0] \ocr[\br][0]}  + \di{0.5\sca}{\ccr \li[0][0] \ocr[\br][0]} +z  \; \di{0.5\sca}{\licu \li[0][0] \ocr[\br][0]} \\
&=-z \; \di{0.5\sca}{\cu \dli[2*\br][0] \li[0][0] \ocr[\br][0]}  + \di{0.5\sca}{\culi} + z \; \di{0.5\sca}{ \ccr}-zv \;\di{0.5\sca}{\licu} +zv\;  \di{0.5\sca}{\licu } \\
&= \di{0.5\sca}{\culi}\; .
\end{align*}
We conclude that sliding is equivalent to the right tangle relation. 
We also have, using sliding and straightening in $\cC$,
\begin{align*}
\di{0.5\sca}{\crc}&= \di{0.5\sca}{\crc \cu[-\br][0] \ca[-2*\br][0]} 
= \di{0.5\sca}{\crc \cu[-\br][0] \ca[-2*\br][\h] \li[-\br][0]}  \\
&= \di{0.5\sca}{\ocr[-\br][0] \cu  \li[\br][0] \ca[-2*\br][\h]  \ca[\br][\h] \li[2*\br][0]} +z\; \di{0.5\sca}{\cu[-\br][\h] \uli[0][\h] \dli[-2*\br][\h]  \ca[-2*\br][\h]  \ca[\br][\h] \dli[\br][\h] \dli[2*\br][\h] }
-z\; \di{0.5\sca}{\cu[0][\h] \uli[0][\h]  \dli[-\br][\h] \dli[-2*\br][\h] \ca[-2*\br][\h]  \ca[\br][\h] \dli[2*\br][\h]} \\
&=\di{0.5\sca}{\ca[-\br][\h] \li[-\br][0] \ocr} + z\; \di{0.5\sca}{\lica} -z\; \di{0.5\sca}{\cali}
\end{align*}
We conclude that upside-down sliding follows from sliding and straightening. 
Similarly, we can show that multiplying upside-down sliding with $\di{0.5\sca}{\ocr \li[2*\br][0]}$ gives us upside-down pulling, while multiplying sliding with $\di{0.5\sca}{\lica}$ leads to the right delooping relation in $\mathcal{B}$. Hence $\cC$ and $\mathcal{B}$ satisfy the same relations and we conclude that $\mathcal{B}$ and $\cC$ are isomorphic categories. 
\end{proof}

\begin{corollary}
The category $\mathcal{C}=\mathcal{C}_{1,1}^{0,0}(+,1,1)$ where we scaled $\parS $ and $\lambda$ to one, is isomorphic to the Brauer category $\mathcal{B}(+1)$ defined in Section \ref{section example Brauer category}.
\end{corollary}
\begin{proof}
If we specialize to $v=v^{-1}=1$ and $z=0$ in the definition of the BWM-category, we get the Brauer category of Section \ref{section example Brauer category}, while if we take the limit $b$ to zero in $\cC^{b,b}_{\lambda,\parS}(+,1,1)$  we get $\cC^{0,0}_{\lambda,\parS}(+,1,1)$. Using our rescaling, we indeed see that $\mathcal{C}_{1,1}^{0,0}(+,1,1)\cong \mathcal{B}(+1)$.
\end{proof}
Note that the discussion about limits of categories at the end of Section \ref{Section classes of categories} shows that the BWM-category is the only possible deformation of the Brauer category in the framework of diagram categories of Brauer type since  $\lim_{b\to 0} \cC^{b,b}_{1,1}(+,1,1)$ is the only possible limit leading to $\mathcal{C}_{1,1}^{0,0}(+,1,1)$

\subsection{The quantum periplectic Brauer category} \label{Subsection q periplectic}
Ahmed, Grantcharov and Guay introduced in \cite{AGG} algebras $\mathcal{A}_q(n)$  as the centralizer of the quantum periplectic Lie superalgebra $U_q(\mathfrak{p}_m)$ acting on $(\mC^{m|m})^{\otimes n}$, leading to a sort of Schur-Weyl duality. 
\begin{definition}[{\cite[Definition 5.1]{AGG}}]
The periplectic q-Brauer algebra $\mathcal{A}_q(n)$ is the unital associative $\mK(q)$-algebra generated by elements $g_i$ and $e_i$ for $1 \leq i \leq n-1$ satisfying the following relations:
\begin{align*}
(g_i-q)(g_i+q^{-1})& =0, \\ e_i^2 &= 0, \\e_i g_i = - q^{-1} e_i,& \qquad g_i e_i = q e_i  \\
g_i g_j = g_j g_ i,  \qquad g_i e_j &= e_j g_i, \qquad e_i e_j = e_j e_i  \qquad \text{ for } \abs{ i-j} \geq 2 \\
g_i g_{i+1} g_i &= g_{i+1} g_i g_{i+1}, 
\\  e_{i+1} e_i e_{i+1}  = -e_{i+1}, & \qquad e_i e_{i+1} e_i = -e_i,  \\
 g_i e_{i+1} e_i& = -g_{i+1} e_i + (q-q^{-1}) e_{i+1} e_i,  \\
 e_{i+1}e_i g_{i+1}&= -e_{i+1}g_i + (q-q^{-1}) e_{i+1} e_i. 
\end{align*}
\end{definition}
In \cite{RuiSong} Rui and Song introduced a monoidal supercategory which they called the periplectic $q$-Brauer category.  It is a deformation of the periplectic Brauer category. The endomorphism spaces of the periplectic $q$-Brauer category give us the periplectic $q$-Brauer algebras $\mathcal{A}_q(n)$.

\begin{definition}[{\cite[Definition 2.2]{RuiSong}}] \label{Def q periplectic category}
The periplectic $q$-Brauer category $\mathcal{B}$ is a strict monoidal supercategory generated by a single object $\bullet$ and two even morphisms
$\di{0.5\sca}{\ocr} \in \mathcal{B}(2,2)$ and $\di{0.5\sca}{\ucr} \in \mathcal{B}(2,2)$
and two odd morphisms
$
\di{0.5\sca}{\ca} \in \mathcal{B}(0,2)$ and $\di{0.5\sca}{\cu} \in \mathcal{B}(2,0)$
subject to the following defining relations:
\begin{enumerate}
\item The braid relations: $\di{0.5\sca}{\ocr \ucr[0][\h]}=\di{0.5\sca}{\lili} = \di{0.5\sca}{\ucr \ocr[0][\h]}$ and $\di{0.25\sca}{\ocr \li[2*\br][0] \li[0][\h] \ocr[\br][\h] \ocr[0][2*\h] \li[2*\br][2*\h]} = \di{0.25\sca}{\ocr[\br][0] \li[0][0] \li[2*\br][\h] \ocr[0][\h] \ocr[\br][2*\h] \li[0][2*\h]}$,
\item The skein relation:  $\di{0.5\sca}{\ocr}-\di{0.5\sca}{\ucr} = (q-q^{-1})\; \di{0.5\sca}{\lili}$,
\item The snake relations: $\di{0.5\sca}{\ca \dli \cu[\br][0] \uli[2*\br][0]}=\di{0.5\sca}{\li}$ and $\di{0.5\sca}{\uli \cu \ca[\br][0] \dli[2*\br][0]}=-\; \di{0.5\sca}{\li}$
\item The untwisting relations: $\di{0.5\sca}{\cu \ocr} = q\; \di{0.5\sca}{\cu}$ and $\di{0.5\sca}{\cu \li \ocr[\br][0]}=\di{0.5\sca}{\ucr \cu[\br][0] \li[2*\br][0]}$,
\item The loop removing relation: $\di{0.5\sca}{\cu \ca}=0$.
\end{enumerate}
\end{definition}
We will now give a connection between this category and a category of Brauer type. 
\begin{proposition}
Consider the category $\mathcal{C}=\mathcal{C}_{-q^{-1},1}^{q-q^{-1},0}(-,1,1)$ where we scaled $\parS =1$ and $a=1$.  
The category $\mathcal{C}$ is isomorphic to the quantum periplectic Brauer category.
Then also $Hom_\mathcal{C} (n,n) \cong \mathcal{A}_q(n)$.
 Furthermore, if we specialize $q=q^{-1}$ to one, we get the periplectic Brauer category. Hence $\mathcal{C}_{-1,1}^{0,0}(-,1,1)$ is isomorphic to the periplectic Brauer category.  
\end{proposition}
\begin{proof}
The category $\cC=\mathcal{C}_{b-\lambda,\parS }^{b,0}(-,1,1)$ has three independent parameters. When we rescale $\parS $ and $a$ to one we have one independent parameter $\lambda$, which we set equal to $q$. Then the parameters of $\cC$ become
\begin{align*}
\lambda&=q, \quad \lambda'=-q^{-1}, \quad \delta=0, \quad \parS =1,\quad \parS '=-1, \quad a=1, \quad b=q-q^{-1}, \quad c=0, \\
\parq &=-q, \quad d=f'=0, \quad d'=-f=-(q-q^{-1}), \quad e=e'=1, \\
 D&=E=0, \quad D'=-q^2+1, \quad E'=q-q^{-1}, \quad F=F'=1.
\end{align*}
Note that the skein-relation allows us to express the under-cross as a linear combination of the over-cross and $\id_2$. Thus we see that $\mathcal{B}$ and $\cC$ are generated by the same object and the same morphisms. We thus only have to show they satisfy the same relations. 
Note that the first braid relation in $\mathcal{B}$ just expresses that the over-cross is invertible, which we also assume to hold in $\cC$. The second braid relation in $\mathcal{B}$ is the same as the braid relation in $\cC$. The skein relation also holds in $\cC$ by Lemma \ref{Lemma invertible}. The snake relations in $\mathcal{B}$ are the straightening relations in $\cC$. 
The sliding relation in $\cC$ is given by  $\di{0.5\sca}{ \ocr \dli \cu[\br][0] \li[2*\br][0] } = \di{0.5\sca}{\ccr} +(q-q^{-1})\;  \di{0.5\sca}{\licu }$. Using the skein relation this is equivalent to $\di{0.5\sca}{\ucr \cu[\br][0] \li[2*\br][0]}=\di{0.5\sca}{\ccr}$. So the untwisting relations and loop-removing relations of $\mathcal{B}$ are also satisfied in $\cC$. We thus have shown that every relation in $\mathcal{B}$ also holds in $\cC$. We still have to show that the relations upside-down untwisting, delooping, pulling, upside-down sliding and upside-down pulling which hold in $\cC$ are also satisfied in $\mathcal{B}$. 
From  \cite[Lemma 2.3]{RuiSong}, we have that $\di{0.5\sca}{\ucr \ca[0][\h]}=-q\; \di{0.5\sca}{\ca}$, $\di{0.5\sca}{
\cu \ocr[\br][0] \li \ca[0][\h] \dli[2*\br][0] \uli[2*\br][\h] 
}= -q \; \di{0.5\sca}{\li} \; ,$ and $\di{0.5\sca}{\ca \ucr[\br][-\h] \li[0][-\h]}=\di{0.5\sca}{\ocr \li[2*\br][0] \ca[\br][\h]}$ hold in $\mathcal{B}$. This is equivalent to upside-down untwisting, delooping and upside-down sliding in $\cC$. 
The pulling relation in $\cC$ is given by $\di{0.5\sca}{\dli[2*\br][0] \cu \ocr[\br][0] \li \li[2*\br][\h] \ocr[0][\h] }=  \di{0.5\sca}{\licu }$. This relation can be obtained via the untwisting relation $\di{0.5\sca}{\cu \li \ocr[\br][0]}=\di{0.5\sca}{\ucr \cu[\br][0] \li[2*\br][0]}$ in $\mathcal{B}$ by multiplying on the left with $\di{0.5\sca}{\li[2*\br][0] \ocr}$. 
upside-down pulling is a bit more involved. We can combine $\di{0.5\sca}{\ca \ucr[\br][-\h] \li[0][-\h]}=\di{0.5\sca}{\ocr \li[2*\br][0] \ca[\br][\h]}$ with the skein relation to obtain 
\begin{align}\label{equation ocr ucr cap}
\di{0.5\sca}{\ca[0][\h] \li \ocr[\br][0]}-(q-q^{-1})\di{0.5\sca}{\ca \li[2*\br][0]} = \; \di{0.5\sca}{\crc}.
\end{align}
 Multiplying on the right with $\di{0.5\sca}{\ocr \li[2*\br][0]}$ leads to 
\begin{align*}
\di{0.5\sca}{ \ca \ocr[\br][-\h] \li[0][-\h] \li[2*\br][-2*\h] \ocr[0][-2*\h] \uli[2*\br][0] }&= (q-q^{-1}) \di{0.5\sca}{\ocr \ca[0][\h] \li[2*\br][0]}+\di{0.5\sca}{\crc \ocr[0][-\h] \li[2*\br][-\h]} \\
&=(-1+q^{-2}) \di{0.5\sca}{\cali} + \di{0.5\sca}{\lica} + (q-q^{-1}) \di{0.5\sca}{\crc}
\\
&=(1-q^2)\; \di{0.5\sca}{\cali} + (q-q^{-1})\; \di{0.5\sca}{ \li \ocr[\br][0] \ca[0][\h] \dli} + \di{0.5\sca}{\lica},  
\end{align*}
where we used upside-down untwisting and the skein relation for the first step and again Equation \eqref{equation ocr ucr cap} in the second step. This shows that the upside-down pulling also holds in $\mathcal{B}$ and we conclude that $\mathcal{B}$ and $\cC$ are isomorphic categories. This also immediately implies the other statements of the proposition since the periplectic Brauer category is obtained from the periplectic Brauer category by setting $q=1$ and the category $\mathcal{C}_{-\lambda,1}^{0,0}$ is obtained from $\mathcal{C}_{b-\lambda,1}^{b,0}$ by taking the limit $b=0$. 
Note that the isomorphism between $\Hom_\cC(n,n)$ and $\mathcal{A}_q(n)$ is given by mapping 
$s_i = \di{0.5\sca}{\li 
\draw (\br,0.5\h) node[]{$\dots$};
\draw (2*\br,-0.5\h) node[]{$i$};
\draw (4*\br,-0.5\h) node[]{$i+1$};
\draw (4*\br,0.5\h) node[]{$\dots$};
\ocr[2*\br][0] \li[5*\br][0]}$ to $g_i$ and 
$ \di{0.5\sca}{\li 
\draw (\br,0.5\h) node[]{$\dots$};
\draw (2*\br,-0.5\h) node[]{$i$};
\draw (4*\br,-0.5\h) node[]{$i+1$};
\draw (4*\br,0.5\h) node[]{$\dots$};
\ca[2*\br][0] \cu[2*\br][\h] \li[5*\br][0]}$ to $e_i$. 
\end{proof} 
We know that $\lim_{b\to 0} \cC^{b,0}_{b-\lambda,\parS}(-,1,1) = \cC^{0,0}_{-\lambda,\parS}(-,1,1)$. If we rescale $\lambda$ and $\parS$ to one and set $b=q-q^{-1}$, this expresses that the periplectic $q$-Brauer category is a deformation of the periplectic Brauer category. However, if we look at the discussion about limits at the end of Section \ref{Section classes of categories}, we see that we have another deformation $\cC^{0,b}_{b-\lambda,\parS}(-,1,1)$ since $\lim_{b \to 0} \cC^{0,b}_{b-\lambda,\parS}(-,1,1)$ is also equal to $\cC^{0,0}_{-\lambda,\parS}(-,1,1)$. 
This deformation is obtained by replacing the untwisting relation $\di{0.5\sca}{\cu \li \ocr[\br][0]}=\di{0.5\sca}{\ucr \cu[\br][0] \li[2*\br][0]}$ by $\di{0.5\sca}{\cu \li \ucr[\br][0]}=\di{0.5\sca}{\ocr \cu[\br][0] \li[2*\br][0]}$ in Definition \ref{Def q periplectic category} of the periplectic $q$-Brauer category.

We thus get two deformations of the periplectic Brauer category. We have seen in Corollary \ref{Corollary Monoidal opposites} that they are each other monoidal opposites. 
 
\subsection{The q-Brauer algebra}  
Aside from the BWM-algebra, there exists another deformation of the Brauer algebra called the $q$-Brauer algebra introduced by Wenzl \cite{Wenzl}. However, no topological or diagrammatical interpretation of this algebra is known. We will now show that they also do not occur as endomorphisms algebras of a category of Brauer type.

The $q$-Brauer algebra $Br_n(q,r)$ is defined \cite[3.1]{Wenzl} as the algebra  with generators $e$ and $g_1, g_2, \dots, g_{n-1}$ satisfying relations
\begin{enumerate}
\item $g_i^2 = (q-1) g_i + q$, $g_ig_{i+1}g_i=g_{i+1}g_ig_{i+1}$, and $g_ig_j= g_jg_i$ if $|i-j|>1$,
\item $e^2=\frac{r-1}{q-1} e$,
\item $e g_i = g_i e$ for $i>2$, $eg_1 = qe$, $eg_2e=re$ and $eg_2^{-1}e=q^{-1}e$
\item $g_2g_3g_1^{-1}g_2^{-1} e_{(2)}=e_{(2)}=e_{(2)}g_2g_3g_1^{-1}g_2^{-1}$ with $e_{(2)}=eg_2g_3g_1^{-1}g_2^{-1}e$.
\end{enumerate}

We will represent $g_i$ by the Brauer diagram $g_i= \di{0.5\sca}{
\li[-\br][0]
\li[\br][0] 
\draw (0,0.5\h) node[]{$\dots$}; 
\draw (2*\br,-0.5\h) node[]{$i$};
\draw (4*\br,-0.5\h) node[]{$i+1$};
\li[4*\br][0]
\draw (5*\br,0.5\h) node[]{$\dots$};
\ocr[2*\br][0] \li[6*\br][0]} $ and $e$ by the Brauer diagram $e=\di{0.5\sca}{\cu[0][\h] \ca \li[2*\br][0] \draw (3*\br,0.5\h) node[]{$\dots$}; \li[4*\br][0]} $.

Assume $\cC$ is a category of Brauer type such that the endomorphism algebras $\Hom_{\cC}(n,n)$ are isomorphic to the $q$-Brauer algebras $Br_n(q,r)$. 
Then the relation $g_i^2 = (q-1) g_i + q$ implies that the parameters in $\cC$ satisfy $a=q$, $b=q-1$, $c=0$, while $eg_1 = qe$, $eg_2e=re$ imply $\lambda=q$, $\parq =r$.  Using Lemma \ref{Lemma invertible},   $eg_2^{-1}e=q^{-1}e$ leads to $\delta=\frac{r-1}{q-1}$, which is also consistent with $e^2=\frac{r-1}{q-1} e$. Since $b$, $\delta$ and $\parq$ are non-zero, but $c=0$, we see from Table \ref{summary table}, that the only possible categories are $\cC^{b,0}_{\lambda,\parS }(-,1,1)$ or $\cC^{b,0}_{\lambda,\parS }(-,-1,-1)$.
Looking at the values of $\parq =-\epsilon e\parS (\lambda-b)$ and $\delta=-\epsilon e \parS (2\lambda-b)/b$ for these categories, we see that they are not compatible with $\delta=\frac{r-1}{q-1}$ and $\parq =r$. We conclude that the $q$-Brauer algebra $Br_n(q,r)$ can not be obtained via a category of Brauer type. 

\appendix
\begingroup
\allowdisplaybreaks 

\section{The equations for simplifying diagrams} \label{Appendix rewrite equations}
In Section \ref{Section rewritable diagrams} we listed all the diagrams we can rewrite in two different ways. In this section, we will deduce the corresponding equations which have to be satisfied such that rewriting is consistent. We will also already simplify the resulting equations. 
\begin{lemma}\label{Lemma S}
Rewriting $\di{0.5\sca}{\cu \ca[\br][0] \uli \dli[2*\br][0] \cu[2*\br][-0.4\h] \uli[3*\br][0] \dli[3*\br][0] }$ and $ 
\di{0.5\sca,-0.5\sca}{\cu \ca[\br][0] \uli \dli[2*\br][0] \cu[2*\br][-0.4\h] \uli[3*\br][0] \dli[3*\br][0] }$ leads to 
\begin{align*}
\parS' = \epsilon \parS .
\end{align*}
\end{lemma}
\begin{proof}
Directly using the upside-down straightening relation on $\di{0.5\sca}{\cu \ca[\br][0] \uli \dli[2*\br][0] \cu[2*\br][-0.4\h] \uli[3*\br][0] \dli[3*\br][0] }$ leads to $\parS' \;\di{0.5\sca}{\cu}$. On the other hand, using the super interchange law and then the straightening law, we have
\[
\di{0.5\sca}{\cu \ca[\br][0] \uli \dli[2*\br][0] \cu[2*\br][-0.4\h] \uli[3*\br][0] \dli[3*\br][0] } = \epsilon\;  \di{0.5\sca}{\cu[2*\br][0] \ca[\br][0] \uli \dli[\br][0] \cu[0][-0.4\h] \uli[3*\br][0] \dli[0][0] } =\epsilon \parS \; \di{0.5\sca}{\cu}.
\]
We conclude that $\parS '=\epsilon \parS $. Similarly, rewriting $ 
\di{0.5\sca,-0.5\sca}{\cu \ca[\br][0] \uli \dli[2*\br][0] \cu[2*\br][-0.4\h] \uli[3*\br][0] \dli[3*\br][0] }$ leads to $\parS = \epsilon \parS '$. 
\end{proof}

\begin{lemma}
Rewriting $\di{0.5\sca}{\dli \ocr \cu[\br][0] \ca[2*\br][0] \dli[3*\br][0]}$ and $\di{0.5\sca,-0.5\sca}{\dli \ucr \cu[\br][0] \ca[2*\br][0] \dli[3*\br][0]}$ gives us the equations
\begin{align*}
\parS (1-ee')=0, \quad \parS '(1-ee')=0, \quad d+ef'=0,  \quad d'+e'f=0, \\ 
 \parS '(f+ed')=0,  \quad \parS (f'+e'd)=0.
\end{align*}
\end{lemma}
\begin{proof}

Using upside-down sliding we deduce 
\[
\di{0.5\sca}{\cu \li \uli[\br][\h] \uli[0][\h] \ocr[\br][0] \ca[2*\br][\h] \li[3*\br][0] \dli[3*\br][0] \dli[2*\br][0]} = d' \parS '\; \di{0.5\sca}{\lili} + e' \parS' \;\di{0.5\sca}{\ocr} + f' \epsilon\;\di{0.5\sca}{\cc}.
\]
This allows us to rewrite $\di{0.5\sca}{\ocr \dli \cu[\br][0] \ca[2*\br][0] \dli[3*\br][0]}$ using sliding as
\begin{align*}
 \di{0.5\sca}{\ocr \dli \cu[\br][0] \ca[2*\br][0] \dli[3*\br][0]} &=d\epsilon \;\di{0.5\sca}{\cc} + e \; \di{0.5\sca}{\cu \li \uli[\br][\h] \uli[0][\h] \ocr[\br][0] \ca[2*\br][\h] \li[3*\br][0] \dli[3*\br][0] \dli[2*\br][0]}  + f \parS' \; \di{0.5\sca}{\lili} \\
 &= (f+ed')\parS' \; \di{0.5\sca}{\lili} + ee'\parS' \;\di{0.5\sca}{\ocr} + (d+ef')\epsilon \;\di{0.5\sca}{\cc}.
\end{align*}
On the other hand, using the straightening relation we obtain $\di{0.5\sca}{\dli \ocr \cu[\br][0] \ca[2*\br][0] \dli[3*\br][0]} = \parS '\; \di{0.5\sca}{\ocr}$.
Comparing these two different rewritings, we obtain
$\parS '(1-ee')=0$, $d+ef'=0$ and $\parS '(f+ed')=0.$ Similarly, rewriting $\di{0.5\sca,-0.5\sca}{\dli \ucr \cu[\br][0] \ca[2*\br][0] \dli[3*\br][0]}$  leads to $\parS (1-ee')=0$, $d'+e'f=0$ and $\parS (f'+e'd)=0.$
\end{proof}
Note that the last two equations are satisfied if the first four equations hold.
Namely multiplying $d'+e'f=0$ with $\parS 'e$ and using $\parS '(1-ee')=0$ we obtain $\parS '(ed'+f)=0$ while multiplying $d+ef'$ with $\parS e'$ gives us $\parS (f'+e'd)=0$. Also, since $\parS '=\epsilon \parS $, the first two equations are equivalent.

\begin{lemma} \label{Lemma a,b,c}
Rewriting $\di{0.5\sca}{\cu \ocr \ocr[0][\h]}\;$, $\di{0.5\sca}{\ocr \ocr[0][\h] \ca[0][2*\h]}\;$,  $\di{0.5\sca}{\cu \ocr \ca[0][\h]}$ and $\di{0.5\sca}{\ocr \ocr[0][\h] \ocr[0][2*\h]}$  gives us 
\[
\lambda^2=a+b\lambda+c\delta  \qquad \lambda'^2 = a+b\lambda' +c\delta \qquad (\lambda'-\lambda) \delta=0, \qquad (\lambda'-\lambda)c=0.
\]
\end{lemma}

\begin{proof}
Rewriting $\di{0.5\sca}{\cu \ocr \ocr[0][\h]}$ using on the one hand twisting and on the other hand untwisting gives us $(a+b\lambda+c\delta)\; \di{0.5\sca}{\cu}=\lambda^2 \; \di{0.5\sca}{\cu}$. Rewriting $\di{0.5\sca}{\ocr \ocr[0][\h] \ca[0][2*\h]}$ leads to $\lambda'^2 = a+b\lambda' +c\delta$. 

We can rewrite $\di{0.5\sca}{\cu \ocr \ca[0][\h]}$ in two different ways. If we first use upside-down untwisting we obtain $\lambda' \;\di{0.5\sca}{\cu  \ca}$ while first using untwisting gives us $\lambda \;\di{0.5\sca}{\cu  \ca}$. So, we conclude that $(\lambda'-\lambda)\delta=0$. 
Rewriting in a similar fashion $\di{0.5\sca}{\ocr \ocr[0][\h] \ocr[0][2*\h]}$ we obtain $(\lambda'-\lambda)c=0$. 
\end{proof}

\begin{corollary}\label{Lemma quadratic equation}
We have two disjoint cases.
If $\lambda'=\lambda$, then $a=\lambda^2-b\lambda-c\delta$. Otherwise, if $\lambda'\not= \lambda$, then $\lambda'=b- \lambda $, $a=-\lambda'  \lambda$, $\delta=0$ and $c=0$. 
\end{corollary}
 \begin{proof}
Subtracting the  equations $\lambda^2=a+b\lambda+c\delta$ and $\lambda'^2 = a+b\lambda' +c\delta$ from each other, we obtain $(\lambda-\lambda')(\lambda+\lambda'-b)=0$. We thus indeed obtain two cases: $\lambda'=\lambda$ or $\lambda'=b-\lambda$. When $\lambda\not=\lambda'$, the equations $(\lambda'-\lambda) \delta=0$ and $(\lambda'-\lambda)c=0$ imply that $\delta=c=0$.
 The expression for $a$ is obtained by simplifying $a= \lambda^2-b\lambda -c\delta$. 
 \end{proof}

\begin{lemma}\label{lemma expression for DEF}
Rewriting $\di{.5\sca}{\dli \ocr \cu[\br][0] \li[2*\br][0] \li[2*\br][\h] \ocr[0][\h]}$ and $\di{.5\sca,-.5\sca}{\dli \ucr \cu[\br][0] \li[2*\br][0] \li[2*\br][\h] \ocr[0][\h]}$ leads to the following expressions for $D,E,F$ and $D',E',F'$:
\begin{align*}
eD= c\parS + (b-\lambda-f)d,\quad eE=e(b-f),  \quad eF=(b-f)f+a, \\
e'D'= c\parS' + (b-\lambda'-f')d',\quad e'E'=e'(b-f'),  \quad e'F'=(b-f')f'+a. 
\end{align*}
\end{lemma}
\begin{proof}
Rewriting $\di{.5\sca}{\dli \ocr \cu[\br][0] \li[2*\br][0] \li[2*\br][\h] \ocr[0][\h]}$ using twisting  leads to 
\[
(bd+c\parS )\; \di{0.5\sca}{\culi} + be\; \di{0.5\sca}{\ccr}+(bf+a) \;\di{0.5\sca}{\licu},
\]
while applying sliding gives us
\[
(d\lambda + eD +fd)\; \di{0.5\sca}{\culi} + (eE+fe)\; \di{0.5\sca}{\ccr}+(eF+f^2) \;\di{0.5\sca}{\licu}.
\]
Rewriting $\di{.5\sca,-.5\sca}{\dli \ucr \cu[\br][0] \li[2*\br][0] \li[2*\br][\h] \ocr[0][\h]}$ gives us similar expressions for $D',E',F'$. 
\end{proof}

\begin{lemma}
Rewriting $\di{0.5\sca}{\dli \ocr \dli[\br][0] \cu[2*\br][0] \li[2*\br][0] \li[3*\br][0] \li[0][\h]  \ocr[\br][\h] \li[3*\br][\h] \ocr[0][2*\h] \li[3*\br][2*\h] \li[2*\br][2*\h]}$  gives us
\begin{align*}
f(b-f)&=0,  \quad  c(f+\epsilon ed)=0 \quad  d(d+eb)=0,  \\
 c\parS 'ed+df\lambda+efD&=0, \quad efE=0, \quad f(a-eF)=0,
\end{align*}
while rewriting $\di{0.5\sca,-0.5\sca}{\dli \ucr \dli[\br][0] \cu[2*\br][0] \li[2*\br][0] \li[3*\br][0] \li[0][\h]  \ucr[\br][\h] \li[3*\br][\h] \ucr[0][2*\h] \li[3*\br][2*\h] \li[2*\br][2*\h]}$ leads to
\begin{align*}
f'(b-f')&=0,  \quad  c(f'+\epsilon e'd')=0 \quad  d'(d'+e'b)=0,  \\
 c\parS e'd'+d'f'\lambda'+e'f'D'&=0, \quad e'f'E'=0, \quad f'(a-e'F')=0.
\end{align*}
\end{lemma}

\begin{proof}
Rewriting $\di{0.5\sca}{\dli \ocr \dli[\br][0] \cu[2*\br][\h]   \li[0][\h]  \ocr[\br][\h] \li[3*\br][\h] \ocr[0][2*\h] \li[3*\br][2*\h] \li[2*\br][2*\h]}$ using sliding leads eventually to
\begin{align*}
&(d^2+edb) \; \di{0.5\sca}{\cu[0][\h] \ocr[2*\br][0]} + de \; \di{0.5\sca}{\cu \ocr[\br][0] \ocr[2*\br][-\h] \li \li[3*\br][0]} + df \; \di{0.5\sca}{\cu \ocr[0][-2*\h] \draw (-\br,0) to (0,-\h); \draw (2*\br,0) to (\br,-\h);} + eda \; \di{0.5\sca}{\cu \lili[2*\br][-\h]} \\
+& e^2 \; \di{0.5\sca}{\cu[0][\h] \ocr[2*\br][0] \li[0][\h] \ocr[\br][\h] \li[3*\br][\h] \li[0][2*\h] \li[\br][2*\h] \ocr[2*\br][2*\h]}+ef \; \di{0.5\sca}{\cu \ocr[\br][0] \li  \ocr[0][-2*\h] \draw (-\br,0) to (0,-\h); \draw (2*\br,0) to (\br,-\h); \li[-\br][0]} +fa \;\di{0.5\sca}{\lili \cu[2*\br][\h] } +fb \di{0.5\sca}{\ocr \cu[2*\br][\h]} 
+ (\epsilon edc+ fc) \; \di{0.5\sca}{\cu[2*\br][0] \dli[0][0] \dli[\br][0] \cc[0][-1.4\h]},
\end{align*}
while rewriting using the braid relation leads to
\begin{align*}
& (d^2+edb) \; \di{0.5\sca}{\cu \li \ocr[\br][0] \li[3*\br][0]} + de \; \di{0.5\sca}{\cu \ocr[\br][0] \ocr[2*\br][-\h] \li \li[3*\br][0]} + df \; \di{0.5\sca}{\cu \ocr[0][-2*\h] \draw (-\br,0) to (0,-\h); \draw (2*\br,0) to (\br,-\h);} 
+ eda \; \di{0.5\sca}{\cu \lili[2*\br][-\h]} 
+ e^2 \; \di{0.5\sca}{\cu[0][\h] \ocr[2*\br][0] \li[0][\h] \ocr[\br][\h] \li[3*\br][\h] \li[0][2*\h] \li[\br][2*\h] \ocr[2*\br][2*\h]}
\\
+& ef \; \di{0.5\sca}{\cu \ocr[\br][0] \li  \ocr[0][-2*\h] \draw (-\br,0) to (0,-\h); \draw (2*\br,0) to (\br,-\h); \li[-\br][0]} +efF \;\di{0.5\sca}{\lili \cu[2*\br][\h] } +f^2 \di{0.5\sca}{\ocr \cu[2*\br][\h]} 
+ (dec\parS '+df\lambda+efD)\; \di{0.5\sca}{\li \cu[\br][\h] \li[3*\br][0]} + efE\; \di{0.5\sca}{\li \ccr[\br][\h]}.
\end{align*}
Comparing these two expressions gives us the first part of the lemma. The calculation for the flipped diagram is similar. 
\end{proof}
Using $d=-ef'$, $d'=-e'f$ and the expression for $D, E, F, D', E'$, and $F'$ from Lemma \ref{lemma expression for DEF} the equations of the previous lemma are equivalent to
\begin{align*}
f(b-f)&=0,  \quad  c(f-e^2\epsilon f')=0, \quad e^2f' (b-f')=0,   \\
c\parS (-e^2\epsilon f' +f)-eff'(b-f)&=0,   \quad   ef(b-f)=0,  \quad  f^2(b-f)=0,
\end{align*}
and
\begin{align*}
f'(b-f')&=0,  \quad   c(f'-e'^2\epsilon f)=0,  \quad e'^2f(b-f)=0,   \\
c\parS '(-e'^2\epsilon f +f')-e'f'f(b-f')&=0,   \quad  e'f'(b-f')=0,  \quad  f'^2(b-f')=0.
\end{align*}

Note that $f(b-f)=0$, $f'(b-f')=0$, $c(f-e^2\epsilon f')=0$ and $c(f'-e'^2\epsilon f)=0$ imply that the other equations hold. 

If we assume $\di{0.5\sca}{\ocr}$ is invertible, then $\lambda$, $\lambda'$ and $a$ are non-zero by Lemma \ref{Lemma invertible}. Combining $eF=(b-f)f+a$ with $f(b-f)=0$, we can conclude that $e$ and $F$ are also non-zero. From Lemma \ref{lemma expression for DEF}, we then also infer that $E = b-f$.

\begin{lemma} \label{Lemma c DEF zero}
Rewriting  $\di{0.5\sca}{\ocr \li[2*\br][0] \ocr[0][\h] \ocr[0][3*\h] \ocr[\br][2*\h] \li[2*\br][\h] \li[2*\br][3*\h] \li[0][2*\h] }$
 and $ 
\di{0.5\sca,-0.5\sca}{\ucr \li[2*\br][0] \ucr[0][\h] \ucr[0][3*\h] \ucr[\br][2*\h] \li[2*\br][\h] \li[2*\br][3*\h] \li[0][2*\h] } $  leads to 
\begin{align*}
cD&=cD'=0, & c(d+eb)&=c(d'+e'b) =0,\\
cE &= cE'=0, & c(ec\parS '+f\lambda) &= c(e'c\parS +f'\lambda')=0,\\
cF&=ce'a,  & cF'&=cea.
\end{align*}
\end{lemma}
\begin{proof}
First applying twisting and then braiding and pulling gives us
\[
\di{0.5\sca}{\ocr \li[2*\br][0] \ocr[0][\h] \ocr[0][3*\h] \ocr[\br][2*\h] \li[2*\br][\h] \li[2*\br][3*\h] \li[0][2*\h] } = a \; \di{0.5\sca}{\li \ocr[\br][0] \ocr[0][\h] \li[2*\br][\h]}+b \di{0.5\sca}{\ocr[\br][0] \li[0][0] \li[2*\br][\h] \ocr[0][\h] \ocr[\br][2*\h] \li[0][2*\h]}\;+ cD\; \di{0.5\sca}{\ca \cu[0][\h] \li[2*\br][0]}+cE \; \di{0.5\sca}{\ca \cu[0][\h] \li[2*\br][0] \li[0][\h] \ocr[\br][\h]}+cF \; \di{0.5\sca}{\ca \lin[0][\h][2*\br][0] \cu[\br][\h]}.
\]
On the other hand, we can first apply braiding twice and then twisting to obtain
\[
\di{0.5\sca}{\ocr \li[2*\br][0] \ocr[0][\h] \ocr[0][3*\h] \ocr[\br][2*\h] \li[2*\br][\h] \li[2*\br][3*\h] \li[0][2*\h] } =  a\; \di{0.5\sca}{\li \ocr[\br][0] \ocr[0][\h] \li[2*\br][\h]}+b \di{0.5\sca}{\ocr[\br][0] \li[0][0] \li[2*\br][\h] \ocr[0][\h] \ocr[\br][2*\h] \li[0][2*\h]}\;+ c \di{0.5\sca}{\li \ocr[\br][0] \ocr[0][\h] \li[2*\br][\h] \li[0][2*\h] \cu[\br][3*\h] \ca[\br][2*\h]}.
\]
This we can rewrite by first using upside-down sliding and then twisting, untwisting and straightening into
\[
 a \;\di{0.5\sca}{\li \ocr[\br][0] \ocr[0][\h] \li[2*\br][\h]}+b \di{0.5\sca}{\ocr[\br][0] \li[0][0] \li[2*\br][\h] \ocr[0][\h] \ocr[\br][2*\h] \li[0][2*\h]}\;+ c(d'+e'b) \di{0.5\sca}{\li \ocr[\br][0] \ca[0][\h] \lin[2*\br][\h][0][2*\h] \cu[\br][2*\h]} + ce'a \di{0.5\sca}{\ca \lin[2*\br][0][0][\h] \cu[\br][\h]}+c(f'\lambda'+e'c\parS ) \di{0.5\sca}{\li  \cc[\br][0]}.
\]
Comparing the two different ways of rewriting we obtain the lemma.
\end{proof}
Note that $eF=a$ and $e'F'=a$ implies that the last two equations are equivalent with $c(ee'-1)=0$ if $e$ and $e'$ are non-zero. 
If we substitute $d=-ef'$ in $c(d+eb)$, we get $ce(b-f')$, which is equivalent to $ceE'=0$. Thus $cE'=cE=0$ implies $c(d+eb)=c(d'+e'b) =0$.

\begin{lemma} \label{Lemma E D}
Rewriting  $\di{0.5\sca}{\ocr[0][0] \cu \ocr[0][2*\h] \ocr[\br][\h] \li[0][\h] \li[2*\br][0] \li[2*\br][2*\h] \dli[2*\br][0]}$
 and $ 
\di{0.5\sca,-0.5\sca}{\ucr[0][0] \cu \ucr[0][2*\h] \ucr[\br][\h] \li[0][\h] \li[2*\br][0] \li[2*\br][2*\h] \dli[2*\br][0]} $  leads to 
\begin{align*}
aE &= \lambda D, &  aE' &= \lambda' D', \\
(b-\lambda)E + D &=0, & (b-\lambda')E' + D' &=0, \\
Ec\parS '&=0, & E'c\parS &=0. 
\end{align*}
\end{lemma}
\begin{proof}
Rewriting $\di{0.5\sca}{\ocr[0][0] \cu \ocr[0][2*\h] \ocr[\br][\h] \li[0][\h] \li[2*\br][0] \li[2*\br][2*\h] \dli[2*\br][0]}$ using untwisting and pulling gives us \[\lambda D  \di{0.5\sca}{\culi}+\lambda E \di{0.5\sca}{\ccr} +\lambda F \di{0.5\sca}{\licu}, \]
while applying the braid relation and pulling leads to
\[
D  \di{0.5\sca}{\culi \li[0][0] \ocr[\br][0]}+E \di{0.5\sca}{\ccr \li[0][0] \ocr[\br][0]} +\di{0.5\sca}{\licu \li[0][0] \ocr[\br][0]}.
\]
The last equation can be simplified using twisting and untwisting to
\[
E a \di{0.5\sca}{\culi}+(Eb+D) \di{0.5\sca}{\ccr} +(Ec\parS '+\lambda F) \di{0.5\sca}{\licu}.
\]
Comparing these two results gives us the lemma.
\end{proof}
If we multiply  $(b-\lambda)E + D =0$ with $\lambda$ and use $(a+b\lambda+c\delta)=\lambda^2$ and $aE=\lambda D$, we conclude that for $\lambda$ non-zero, the equation $(b-\lambda)E + D =0$ holds if $c\delta E=0$ holds. This in turn follows from Lemma \ref{Lemma c DEF zero} which says that we already have the stronger $cE=0$.
Hence, the first equation in Lemma \ref{Lemma E D} combined with Lemma \ref{Lemma a,b,c}  and Lemma \ref{Lemma c DEF zero} implies the last two equations of Lemma \ref{Lemma E D}. 

To summarize, we found that rewriting the diagrams in Lemma \ref{Lemma S} to Lemma \ref{Lemma E D} leads to the following equations if $\di{0.5\sca}{\ocr}$ is invertible: 
\begin{equation}\label{First equation}
\begin{aligned}
\parS '&=\epsilon \parS \\
\parS ee'&=\parS \\
d&=-ef' & d'&=-e'f \\
E&=b-f &
E'&=b-f' \\
F&= \frac{a}{e} &
F'&=\frac{a}{e'}\\
D&= \frac{c\parS }{e}-(b-\lambda-f)f' &
D'&= \frac{c\parS '}{e'}-(b-\lambda'-f')f\\
f(b-f) &= 0 & f'(b-f') &= 0  \\
c(f-e^2\epsilon f')&=0 & c(f'-e'^2\epsilon f)&=0 \\
D&= \frac{aE}{\lambda} & D'& = \frac{aE'}{\lambda'} \\
cE&=0 & cE'&=0 \\
cee'&=c \\
c(ec\parS '+f\lambda)&=0 & c(e'c\parS +f'\lambda')&=0.
\end{aligned}
\end{equation}
If $\lambda=\lambda'$, we have
\begin{align}
a=\lambda^2-b\lambda-c\delta, 
\end{align}
while if $\lambda\not=\lambda'$, we have
\begin{align}
a=-\lambda'\lambda, \quad \lambda'=b-\lambda, \quad \delta=c=0.
\end{align}
We will make frequent use of these equations to simplify the equations we derive for the other rewritable diagrams.
From now on, we will also no longer give proofs of the rewriting lemmas. They can be obtained in a similar manner to the proofs of Lemma \ref{Lemma S} to Lemma \ref{Lemma E D}.

\begin{lemma}
Rewriting  $\di{0.5\sca}{\ocr \li[2*\br][0] \li[0][\h] \ocr[\br][\h] \li[2*\br][2*\h] \ocr[0][2*\h] \li[0][3*\h] \ocr[\br][3*\h] \li[2*\br][4*\h] \ocr[0][4*\h]  } $  leads to 
\begin{align*}
cead'&=ce'ad, \\
ca(ee'-1)&=c(d(d'+e'b))= c(d'(d+eb)),\\
ceaf'&=c(bea+d(e'c\parS +f'\lambda')), \\
ce'af &=  c(be'a+d'(ec\parS '+f\lambda)),\\
ce(d'+e'b)&=ce'(d+eb), \\
 c(d+eb)f' &= c(bd + b^2e +e(e'c\parS +f'\lambda')), \\
c(d'+e'b)f&= c(bd'+b^2e' +e'(ec\parS '+f\lambda)), \\
 c(bc e'\parS +bf'\lambda' +(ec\parS '+f\lambda)f')&= c(bec\parS '+bf\lambda + f(e'c\parS +f'\lambda')).
\end{align*}
\end{lemma}
Note that using $d=-ef'$, $d'=-ef$, $c(b-f')=c(b-f)=0$, $c(ee'-1)=0$,  $c(ec\parS '+f\lambda)=0$ and  $c(e'c\parS +f'\lambda')=0$, we can easily verify that all equations but the first are always trivially satisfied. 
Using the fact that $a$, $e$ and $e'$ are non-zero, the first equation is equivalent to 
\begin{align}
cf=cf'. 
\end{align}

\begin{lemma}
Rewriting  $\di{0.5\sca}{\cu \li \ocr[\br][0] \ocr[0][\h] \ocr[0][2*\h]   \li[2*\br][\h] \li[2*\br][2*\h] \dli[2*\br][0] } $
 and $ 
\di{0.5\sca,-0.5\sca}{\cu \li \ucr[\br][0] \ucr[0][\h] \ucr[0][2*\h]   \li[2*\br][\h] \li[2*\br][2*\h] \dli[2*\br][0] }$  leads to 
\begin{align*}
c\parq  &= D(\lambda+E-b) + Fd =D'(\lambda'+E'-b) + F'd', \\
a &= E(E-b)+Fe = E'(E'-b)+F'e,' \\
0&=F(E+f-b)=F'(E'+f'-b).
\end{align*}
\end{lemma}
Note that the last equation is satisfied since $E=b-f$ while $a=Fe$ and $(b-f)f=0$ implies that the second equation is satisfied. 
On the other hand, the first equation is equivalent to 
\begin{align}
c\parq  = a(b-f-f'),
\end{align}
by using $D(E-b) =(aE/\lambda)( -f) = 0$.

\begin{lemma}
Rewriting  $\di{0.5\sca}{ \ocr[0][0] \ocr[\br][\h] \ca[0][2*\h] \li[0][\h] \uli[2*\br][2*\h] \li[3*\br][\h]  \uli[3*\br][2*\h] \cu[2*\br][\h]  }  $
 and $\di{0.5\sca,-0.5\sca}{ \ucr[0][0] \ucr[\br][\h] \ca[0][2*\h] \li[0][\h] \uli[2*\br][2*\h] \li[3*\br][\h]  \uli[3*\br][2*\h] \cu[2*\br][\h]  } $  leads to 
\begin{align*}
\parS' ae' &= \parS '(F+Ed'), & \parS ae &= \parS (F'+E'd), \\
\parS' (d'+e'b)&= Ee'\parS', & \parS (d+eb)&= E'e\parS, \\
f'\lambda + \epsilon \parS 'e'c &= Ef'+D,  & f\lambda' + \epsilon \parS ec &= E'f+D'.
\end{align*}
\end{lemma}
The first two equations are always satisfied.
The last equation is equivalent to
\begin{align}
\lambda (\parS e'c +ff')+c\delta f'&= \lambda' (\parS' ec+ff')+c\delta f= a(b-f-f'),
\end{align}
where we used $D=a (b-f)/\lambda$ and $\lambda^2-b\lambda=a+c\delta$.

\begin{lemma}\label{lemma q S delta}
Rewriting  $\di{0.5\sca}{\cu \li \ca[0][\h] \ocr[\br][0] \uli[2*\br][\h] \dli[2*\br][0] \cu[2*\br][-0.4\h] \dli[3*\br][0] \li[3*\br][0] \uli[3*\br][\h]  }  $
 and $\di{0.5\sca,-0.5\sca}{\cu \li \ca[0][\h] \ucr[\br][0] \uli[2*\br][\h] \dli[2*\br][0] \cu[2*\br][-0.4\h] \dli[3*\br][0] \li[3*\br][0] \uli[3*\br][\h]  } $  leads to 
\begin{align}
\parq   = \epsilon \parS (d+e\lambda)+f\delta= \epsilon \parS '(d'+e'\lambda') + f'\delta.
\end{align}
\end{lemma}

\begin{lemma}\label{lemma q S delta 2}
Rewriting $\di{0.5\sca,-0.5\sca}{\cu \ucr[0][0] \ca[\br][\h] \uli[0][\h] \dli[2*\br][0] \li[2*\br][0]}$ and $\di{0.5\sca}{\cu \ocr[0][0] \ca[\br][\h] \uli[0][\h] \dli[2*\br][0] \li[2*\br][0]} $ leads to 
\[
e\parq   = -d\delta +  (\lambda'-f)\parS \quad  \text{ and }  \quad  e'\parq  = -d'\delta +(\lambda-f')\parS '.
\]
\end{lemma}
From $d= -ef'$ and $\parS (1-ee')=0$, we see that the equation in Lemma \ref{lemma q S delta 2} is equivalent to Lemma \ref{lemma q S delta}.

\begin{lemma}
Rewriting  $ \di{0.5\sca}{\ocr[0][0]\dli \cu[\br][0] \li[2*\br][0] \ca[\br][\h] \uli[0][\h]}  $  leads to 
\[
 \parS '(d+e\lambda')+f\delta= \parS (d'+e'\lambda) + f'\delta.
\]
\end{lemma}
Using Lemma \ref{lemma q S delta}, we can rewrite this equation as
\begin{align}
\parS (\lambda'-\lambda)(1+\epsilon e^2)=0. 
\end{align}

\begin{lemma}
Rewriting  $\di{0.5\sca}{\ocr[0][0] \cu \ca[0][2*\h] \ocr[\br][\h] \li[0][\h] \li[2*\br][0] \uli[2*\br][2*\h] \dli[2*\br][0]}  $
 and $\di{0.5\sca,-0.5\sca}{\ucr[0][0] \cu \ca[0][2*\h] \ucr[\br][\h] \li[0][\h] \li[2*\br][0] \uli[2*\br][2*\h] \dli[2*\br][0]} $  leads to 
\[
(\lambda-E')\parq  = D'\delta + F'\parS ' \quad \text{ and } \quad  (\lambda'-E)\parq   = D\delta + F\parS .
\]
\end{lemma}
This is equivalent to 
\begin{align}
(\lambda-b+f')\parq  = \frac{a}{\lambda'}(b-f')\delta+a\parS 'e \quad \text{ and } \quad (\lambda'-b+f)\parq = \frac{a}{\lambda}(b-f) \delta + a\parS e'. 
\end{align}
\begin{lemma}
Rewriting  $\di{0.5\sca}{\cu \dli[2*\br][0] \li \ocr[\br][0] \ocr[0][\h] \li[2*\br][\h] \li[0][2*\h] \ocr[\br][2*\h] \ca[0][3*\h] \uli[2*\br][3*\h]} $  leads to 
\[
\parq  (D+Eb) + a\delta E +\parS (c\parS 'E+F\lambda) = \parq (D'+E'b) + a\delta E' + \parS '(c\parS E'+ F'\lambda').
\]
\end{lemma}
Using $cE=0$ and $a/\lambda+b= \lambda-c\delta/\lambda$, this is equivalent to
\begin{align}
(b-f)(\parq \lambda+a\delta ) + \parS a\lambda e'  =(b-f')(\parq \lambda'+a\delta ) + \parS 'a\lambda' e .
\end{align}

\begin{lemma}
Rewriting  $\di{0.5\sca}{ \ocr[0][0] \ocr[\br][\h] \ca[0][2*\h] \li[0][\h] \uli[2*\br][2*\h] \li[2*\br][0] \cu[\br][0] \dli }  $
 and $\di{0.5\sca,-0.5\sca}{ \ucr[0][0] \ucr[\br][\h] \ca[0][2*\h] \li[0][\h] \uli[2*\br][2*\h] \li[2*\br][0] \cu[\br][0] \dli } $  leads to 
\begin{align*}
\parS (D'+\lambda E'-f\lambda-ec\parS ') &= (-F'+ea)\delta+(d+eb)\parq  ,\\
\parS '(D+\lambda' E-f'\lambda'-e'c\parS ) &= (-F+e'a)\delta+(d'+e'b)\parq  .
\end{align*}
\end{lemma}
This is equivalent to
\begin{equation}
\begin{aligned}
(b-f')(\parS \frac{a}{\lambda'} -e\parq ) +(b-f'-f)\lambda \parS -ec\parS \parS '=a\delta(e-\frac{1}{e'}), \\
(b-f)(\parS' \frac{a}{\lambda} -e'\parq ) +(b-f-f')\lambda' \parS' -e'c\parS \parS '=a\delta(e'-\frac{1}{e}).
\end{aligned}
\end{equation}

\begin{lemma}
Rewriting  $\di{0.5\sca}{\cu \li \ocr[\br][0] \li[2*\br][\h] \ocr[0][\h] \dli[2*\br][0] \dli[3*\br][0] \cu[2*\br][-0.4\h] \li[3*\br][0] \li[3*\br][\h]}  $
 and $ 
\di{0.5\sca,-0.5\sca}{\cu \li \ucr[\br][0] \li[2*\br][\h] \ucr[0][\h] \dli[2*\br][0] \dli[3*\br][0] \cu[2*\br][-0.4\h] \li[3*\br][0] \li[3*\br][\h]}$  leads to 
\begin{align*}
\epsilon D+\epsilon Ef &= d^2+\epsilon f\lambda+def + ed\lambda +e^3c\parS' +e^2f\lambda,  \\
 \epsilon Ed + F &= d^2e + \epsilon df + e^3 a, \\
 \epsilon E e &= de^2 + e^2d + e^3 b + \epsilon ef,  \\
 \epsilon D'+\epsilon E'f' &= d'^2+\epsilon f'\lambda'+d'e'f' + e'd'\lambda' +e'^3c\parS +e'^2f'\lambda',  \\
 \epsilon E'd' + F' &= d'^2e' + \epsilon d'f' + e'^3 a, \\
 \epsilon E' e' &= d'e'^2 + e'^2d' + e'^3 b + \epsilon  e'f'.
\end{align*}
\end{lemma}
These equations are equivalent to
\begin{align*}
\epsilon(a(b-f)/\lambda -f\lambda) &= e^2(f'^2-ff'-f'\lambda+f\lambda)+e^3 c\parS', \\
\epsilon e^2(2f-b)f'+a&=e^4(a+f'^2), \\
\epsilon e^2(b-2f) &= e^4(b-2f'), \\
\epsilon(a(b-f')/\lambda' -f'\lambda') &= e'^2(f^2-ff'-f\lambda'+f'\lambda')+e'^3 c\parS, \\
\epsilon e'^2(2f'-b)f+a&=e'^4(a+f^2), \\
\epsilon e'^2(b-2f') &= e'^4(b-2f).
\end{align*}
Substituting $\epsilon e^2(b-2f) = e^4(b-2f')$ in $\epsilon e^2(2f-b)f'+a=e^4(a+f'^2)$ gives us that $e^4=1$, so that we can also rewrite these equations as
\begin{equation}
\begin{aligned}
\epsilon e^2(a(b-f)/\lambda -f\lambda) &= (f'^2-ff'-f'\lambda+f\lambda)+e c\parS', \\
\epsilon e'^2(a(b-f')/\lambda' -f'\lambda') &= (f^2-ff'-f\lambda'+f'\lambda')+e'c\parS, \\
e^4&=e'^4=1, \\
\epsilon e^2(b-2f) &= (b-2f'), \\
\epsilon e'^2(b-2f') &= (b-2f).
\end{aligned}
\end{equation}

\begin{lemma}
Rewriting  $\di{0.5\sca}{\dli[0][\h] \cu[\br][\h] \ocr[0][\h] \ocr[0][3*\h] \ocr[\br][2*\h] \li[2*\br][\h] \li[2*\br][3*\h] \li[0][2*\h] }$
 and $ 
\di{0.5\sca,-0.5\sca}{\dli[0][\h] \cu[\br][\h] \ucr[0][\h] \ucr[0][3*\h] \ucr[\br][2*\h] \li[2*\br][\h] \li[2*\br][3*\h] \li[0][2*\h] } $  leads to 
\begin{align*}
 ec\parS' d + f \lambda d &= -D(d+eb), \\
   e'c\parS d' + f' \lambda' d' &= -D'(d'+e'b),\\
\lambda(eb-ef+d)-e^2 c\parS' &= E(d+eb),  \\
\lambda'(e'b-e'f'+d')-e'^2 c\parS &= E'(d'+e'b), \\
\lambda(ec\parS '+f\lambda-f^2) -ec\parS 'f &= F(d+eb), \\
 \lambda'(e'c\parS +f'\lambda'-f'^2) -e'c\parS f' &= F'(d'+e'b).
\end{align*}
\end{lemma}
This we can rewrite as
\begin{equation}
\begin{aligned}
f'(ec\parS '+f\lambda)&= a(b-f)(b-f')/\lambda, \\
f(e'c\parS +f'\lambda')&= a(b-f)(b-f')/\lambda', \\
\lambda(b-f-f') -ec\parS '&=(b-f)(b-f'), \\
\lambda'(b-f'-f) -e'c\parS &=(b-f)(b-f'), \\
\lambda (ec\parS '+f\lambda-f^2) -ec\parS 'f &= a(b-f'),  \\
\lambda' (e'c\parS +f'\lambda'-f'^2) -e'c\parS f' &= a(b-f). 
\end{aligned}
\end{equation}

\begin{lemma}
Rewriting  $\di{0.5\sca}{\dli[2*\br][0] \ocr[\br][0] \li \cu[0][0] \ocr[0][\h] \ocr[0][3*\h] \ocr[\br][2*\h] \li[2*\br][\h]  \li[2*\br][3*\h] \li[0][2*\h] }$ and $
\di{0.5\sca,-0.5\sca}{\dli[2*\br][0] \ucr[\br][0] \li \cu[0][0] \ucr[0][\h] \ucr[0][3*\h] \ucr[\br][2*\h] \li[2*\br][\h]  \li[2*\br][3*\h] \li[0][2*\h] } $  leads to 
\begin{equation} \label{Last equation}
\begin{aligned}
a(bE+c\parS' e )&= D^2+Ea\lambda + EbD+ Ec\parS' d +F\lambda d, \\
a\lambda+bD+b^2E +c\parS' d +c\parS' eb &= DE +bE^2+Ec\parS' e + F\lambda e, \\
b E c\parS' +bF\lambda + c^2 \parS '^2 e + c\parS' \lambda  f &= DF + EbF +Ec\parS'  f + F \lambda f,\\
a(bE'+c\parS e' )&= D'^2+E'a\lambda' + E'bD'+ E'c\parS d' +F'\lambda' d', \\
a\lambda'+bD'+b^2E' +c\parS d' +c\parS e'b &= D'E' +bE'^2+E'c\parS e' + F'\lambda' e', \\
b E' c\parS  +bF'\lambda' + c^2 \parS ^2 e' + c\parS \lambda' f' &= D'F' + E'bF' +E'c\parS f' + F' \lambda' f'.
\end{aligned}
\end{equation}

\end{lemma}
We conclude that simplifying the overlapping diagrams of Section \ref{Section rewritable diagrams} leads to a consistent result if Equations \eqref{First equation} to \eqref{Last equation} are satisfied.

\endgroup

\subsection*{Acknowledgements}
This research has been supported by a FWO postdoctoral junior fellowship from the Research Foundation Flanders (1269821N).

\bibliography{citations} 
\bibliographystyle{alphaabbr}

\end{document}